\documentclass[reqno,11pt, psfig]{amsart}
\usepackage{amsfonts}
\usepackage{amssymb, amsmath,amssymb, amscd, amsbsy,latexsym}
\usepackage{amsthm, amsfonts,mathrsfs}
\usepackage[T1]{fontenc}
\usepackage{times}
\usepackage{epsfig}
\usepackage{color}
\usepackage[all]{xypic}
\input amssym.def
\input amssym.tex

\def\inte#1{
\displaystyle\mathop{#1\kern0pt}^\circ }

\let\pa=\partial

\def\cR{{\mathcal R}}

\def\pa{\partial}

\def\l{\frak l}
\def\virgp{\raise 2pt\hbox{,}}
\def\cdotpv{\raise 2pt\hbox{;}}

\def\eqdefa{\buildrel\hbox{\footnotesize def}\over =}

\def\C{\mathop{\mathbb C\kern 0pt}\nolimits}
\def\DD{\mathop{\mathbb D\kern 0pt}\nolimits}
\def\EE{\mathop{{\mathbb E \kern 0pt}}\nolimits}
\def\K{\mathop{\mathbb K\kern 0pt}\nolimits}
\def\N{\mathop{\mathbb N\kern 0pt}\nolimits}
\def\Q{\mathop{\mathbb Q\kern 0pt}\nolimits}
\def\R{\mathop{\mathbb R\kern 0pt}\nolimits}
\def\SS{\mathop{\mathbb S\kern 0pt}\nolimits}
\def\ZZ{\mathop{\mathbb Z\kern 0pt}\nolimits}
\def\TT{\mathop{\mathbb T\kern 0pt}\nolimits}
\def\P{\mathop{\mathbb P\kern 0pt}\nolimits}

\def\no{\noindent}

\newcommand{\beq}{\begin{equation}}
\newcommand{\eeq}{\end{equation}}
\newcommand{\ben}{\begin{eqnarray}}
\newcommand{\een}{\end{eqnarray}}
\newcommand{\beno}{\begin{eqnarray*}}
\newcommand{\eeno}{\end{eqnarray*}}
\newtheorem{defi}{Definition}[section]
\newtheorem{thm}{Theorem}[section]
\newtheorem{lem}{Lemma}[section]

\renewcommand{\theequation}{\thesection.\arabic{equation}}
\newtheorem{remark}{Remark}[section]
\newtheorem{lemma}{Lemma}[section]
\newtheorem{cor}{Corollary}[section]

\newtheorem{theorem}{Theorem}[section]
 \def\fr#1#2{{\textstyle \frac{#1}{#2}}}

\newdimen\eqjot \eqjot = 1\jot
\def\openupeq{\openup \the\eqjot}

\def\qaeq#1#2{{\def\\{&}\vcenter{\openupeq\halign{$\displaystyle
   ##\hfil$&&\hskip#1pt$\displaystyle##\hfil$\cr #2\cr}}}}

\def\qeq{\qaeq{20}}

\def\pofbox#1 #2$#3${\setbox0=\hbox{$#3$}\ht0=0pt\dp0=0pt\wd0=0pt\hskip-#1pt\raise#2pt\box0\hskip#1pt}

\begin{document}

\title[Blow-up and peakons to a generalized $\mu$-Camassa-Holm equation]
{Blow-up solutions and peakons to a generalized $\mu$-Camassa-Holm integrable equation}

\author{Changzheng Qu}
\address{Changzheng Qu\newline
Department of Mathematics, Ningbo University, Ningbo 315211, P. R.
China} \email{quchangzheng@nbu.edu.cn}
\author{Ying Fu}
\address{Ying Fu\newline
Department of Mathematics\\
Northwest University\\
Xi'an 710069\\
P. R. China} \email{fuying@nwu.edu.cn}
\author{Yue Liu}
\address{Yue Liu\newline
 Department of Mathematics, University of Texas, Arlington, TX 76019-0408} \email{yliu@uta.edu}

%\date{January 17, 2013}
 \maketitle

\begin{abstract}
Consideration here is  a generalized $\mu$-type integrable equation, which can be regarded as a generalization to both the $\mu$-Camassa-Holm and modified $\mu$-Camassa-Holm equations. It is shown that the proposed equation is formally integrable with the Lax-pair and the bi-Hamiltonian structure and its  scale limit is an integrable model of hydrodynamical systems describing short capillary-gravity waves. Local well-posedness of the Cauchy problem in the suitable Sobolev space is established by the viscosity method. Existence of peaked traveling-wave solutions and formation of singularities of solutions for the equation are investigated. It is found that the equation admits a single peaked soliton and multi-peakon solutions. The effects of varying $\mu$-Camassa-Holm and  modified $\mu$-Camassa-Holm nonlocal nonlinearities on blow-up criteria and wave breaking   are illustrated in detail. Our  analysis  relies on the method of characteristics and conserved quantities and is proceeded with a priori differential estimates.
\end{abstract}

\maketitle


\vskip 0.2cm


%%%%%%%%%%%%%%%%%%%%%%%%%%%%%%%%%%%%%%%%%%%%%%%%%%%%%%%%%%%%
\renewcommand{\theequation}{\thesection.\arabic{equation}}
\setcounter{equation}{0}
%%%%%%%%%%%%%%%%%%%%%%%%%%%%%%%%%%%%%%%%%%%%%%%%%%%%%%%%%%%%

\section{Introduction}
In this paper, we are concerned with the following new integrable partial
differential equation
\begin{equation}\label{m-m-mu-ch}
m_t+k_1\big(2\mu(u)u-u^2_x)m\big)_x+k_2(2mu_x+um_x)+\gamma u_x=0,
\end{equation}
where $u(t,x)$ is a function of time $t$ and a single spatial variable
$x$, and
 \begin{equation}\label{mm}
  m=\mu(u)-u_{xx},\quad \mu(u)=
\int_{\Bbb S}u(t,x)dx,
\end{equation}
with ${\Bbb S}={\Bbb R}/{\Bbb Z}$ which denotes the unit circle on ${\Bbb
R}^2$. Eq.\eqref{m-m-mu-ch} is  reduced to the $\mu$-Camassa-Holm (CH) equation \cite{klm}
\begin{equation}\label{mu-ch}
m_t+2mu_x+um_x + \gamma u_x =0,
\end{equation}
for $k_1=0$ and $k_2=1,$ and the modified $\mu$-CH equation \cite{qfl}
\begin{equation}\label{m-mu-ch}
m_t+\big((2\mu(u)u-u^2_x)m\big)_x + \gamma u_x =0,
\end{equation}
for $k_2=0$ and  $k_1=1$, respectively.

It is known that the CH equation \cite{ch, ff} of the following form
\begin{equation}\label{ch}
m_t + u m_x + 2 u_x m+\gamma u_x=0,\quad\text{with}\quad m=u-u_{xx}
\end{equation}
was proposed as a model for the unidirectional propagation of the
shallow water waves over a flat bottom (see also \cite{cl, joh}),
with $u(x,t)$ representing the height of the water's free surface in
terms of non-dimensional variables. It is completely integrable with
a bi-Hamiltonian structure and an infinite number of conservation laws \cite {ch,
ff}. It is of interest to note that the CH equation can also be
derived by tri-Hamiltonian duality from the KdV equation (a number
of additional examples of dual integrable systems derived applying
the method of tri-Hamiltonian duality can be found in
\cite{fuc,or}). The CH equation has two remarkable features:
existence of peakon and multi-peakons \cite{achm, ch,cht} (when $ \gamma
= 0$) and breaking waves, i.e., wave profile remains bounded while
its slope becomes unbounded in finite time \cite{con1,con2,
ce-1,ce-2,ce-3,lo}. Geometrically, the CH equation describes the geodesic flows on the Bott-Virasoro
group for the case $\gamma\not=0$ \cite{mis,sch} and on the
diffeomorphism group of the unit circle under $H^1$ metric for the
case $\gamma=0$ \cite{kou}, respectively. Note that the two geometric descriptions
are not analogous: for $ \gamma  = 0 $ (on the diffeomorphism group)
the Riemannian exponential map is a local chart, but this is not the
case for $\gamma \neq 0$ (on the Bott-Virasoro group) cf. the
discussion in \cite {ckkt}. The CH equation also arises
from a non-stretching invariant planar curve flow in the
centro-equiaffine geometry \cite{cq,olv}. Well-posedness and wave
breaking of the CH equation were studied extensively, and many
interesting results have been obtained, see \cite{con1, ce-1, ce-2,
ce-3, lo}, for example. The $\mu$-CH equation \eqref{mu-ch} was first introduced
by Khesin, Lenells and Misio\l ek \cite{klm}. They verified that the $\mu$-CH equation is bi-Hamiltonian and admits
cusped and smooth traveling wave solutions. In \cite{lmt}, it was shown that the $\mu$-CH equation also admits single peakon and multi-peakons. Interestingly, the $\mu$-CH equation describes a geodesic flow on diffeomorphism
group over ${\Bbb S}$ with certain metric. Its well-posedness and blow up were investigated in \cite{flq,klm}.

Another well-known integrable equation admitting peakons with quadratic nonlinearities is the
Degasperis-Procesi (DP) equation \cite{dp}, which takes the form
\begin{equation}\label{dp}
m_t+um_x+3u_xm=0, \;\; m=u-u_{xx}.
\end{equation}
It is regarded as a model for nonlinear shallow water dynamics, which can also be obtained from the governing
equations for water waves \cite{cl}. Its asymptotic accuracy is the same as for the Camassa-Holm shallow
water equation. More interestingly, it admits
the shock peakons in both periodic \cite{ely-2} and non-periodic
settings \cite{ly}. Wave breaking phenomena and global existence of
solutions for the Degasperis-Procesi equation were investigated in
\cite{ck,dhh, ely-1,ely-2,ly,lun}. The $\mu$-version of the DP equation
was introduced by Lenells, Misio\l ek and Ti\u{g}lay \cite{lmt}, which takes the form \eqref{dp}
with $m$ replaced by $m=\mu(u)-u_{xx}$. Its integrability,
well-posedness, blow up and existence of peakons were investigated in \cite{flq,
lmt}.

It is  noted that all nonlinear terms in the CH and DP equations are quadratic.
So it is of great interest to find such integrable equations with
cubic and higher-order nonlinear terms. The CH equation can be obtained
by the tri-Hamiltonian duality approach from the bi-Hamiltonian structure of the KdV equation \cite{or}. Similarly, the approach
applied to the modified KdV equation then yields the so-called modified CH equation \cite{or,fok,qia}
\begin{equation}\label{m-ch}
m_t+\big((u^2-u^2_x)m\big)_x=0,\quad m=u-u_{xx},
\end{equation}
which has cubic nonlinearities, and is integrable with the  Lax-pair and the bi-Hamiltonian structure.
It was shown that the modified CH equation has a single peaked soliton and multi-peakons with a different character than those of the CH equation \cite{gloq},
and it also has new features of blow-up criterion and wave breaking mechanism. The $\mu$-version of the modified CH equation, called the modified $\mu$-CH equation
\begin{equation}\label{m-mu-ch}
m_t+\big((2u\mu(u)-u^2_x)m\big)_x=0,\quad m=\mu(u)-u_{xx},
\end{equation}
was introduced in \cite{qfl}. It is also formally integrable with the bi-Hamiltonian structure and the Lax-pair, and arises from a non-stretching planar curve flows in ${\mathbb R}^2$ \cite{qfl}. On the other hand, its local well-posedness, wave breaking, existence of peakons and their stability
were discussed in \cite{qfl,lqz}. As an extension of both the CH and modified CH equations, an integrable equation with both quadratic and cubic nonlinearities has been introduced by Fokas \cite{fok}, which takes the form
\begin{equation}\label{fokas-ch}
m_t+k_1 \big((u^2-u_x^2))m\big)_x+k_2(2u_xm+um_x)=0,\quad m=u-u_{xx}.
\end{equation}
It was shown by Qiao et al \cite{qxl} that Eq.\eqref{fokas-ch} is  integrable with the Lax-pair and the bi-Hamiltonian structure. Its peaked solitons were also obtained in \cite{qxl}. Indeed, Eq.\eqref{fokas-ch} can be obtained by the tri-Hamiltonian duality approach from the bi-Hamiltonian structure of the Gardner equation
\begin{equation}\label{garder-eq}
m_t+k_1 u^2u_x+k_2 uu_x=0.
\end{equation}
It was noticed in \cite{gloq} and \cite{qfl} that the short-pulse equation  as an approximation of the Maxwell equation \cite{sw},
\begin{equation}\label{sp-eq}
v_{xt}+\frac 12 (v^2v_x)_x+v=0
\end{equation}
is a scale limit of both the modified CH equation \eqref{m-ch} and the modified $\mu$-CH \eqref{m-mu-ch} equation.

The goal of the present  paper is to investigate whether or not the generalized $\mu$-version \eqref{m-m-mu-ch}  with  two different nonlocal nonlinearities has  similar remarkable properties as the CH equation \eqref{ch} and the modified CH equation \eqref{m-ch} as well as their $\mu$-versions. As is well known, in order for persistent structures and properties to exist in such equations, there must be a delicate nonlinear/nonlocal balance, even in the absence of linear dispersion with $ \gamma = 0.$  Indeed, it is shown in Theorem \ref{thm-peakon-1} that  Eq.\eqref{m-m-mu-ch}  admits a single peaked soliton and multi-peakon solutions even the parameters $ k_1 $ and $ k_2 $ have  a different sign.

On the other hand, it is found that those two parameters $ k_1 $ and $ k_2 $ corresponding to  the modified $\mu$-CH cubic and $\mu$-CH quadratic nonlinearities respectively play an important role in the effects of varying nonlocal nonlinearites on the finite blow-up of solutions, or wave breaking. It is well known that wave breaking relies crucially on strong nonlinear and nonlocal dispersion. Indeed, in the case of the CH equation with the quadratic nonlinearity, the wave {\it cannot} break in finite time when the  momentum potential $ m = (1 - \partial_x^2) u $ keeps positive initially \cite{con1}.  The breaking wave, however, occurs for the modified CH equation with the cubic nonlinearity, even  $ m $ is positive initially \cite {gloq}. Those of key different features of wave breaking make our investigation on the wave breaking to the generalized $\mu$-version  \eqref{m-m-mu-ch} more interesting but the analysis more  challenging.

In view of the transport theory, the key issue to obtain the wave breaking result (Theorem \ref{thm-7.2}) is to derive a priori differential estimate so that the slope $ k_1 u_x m + k_2 u_x $ is unbounded below. Rather, the nonlocal nature and mixed strong nonlinearites of the problem prohibit the use of classical blow-up approaches. To this end, we shall adopt the
method of characteristics with conservative quantities and the new mechanism of a priori differential inequalities to control higher nonlocal nonlinearities. This new approach is expected to have more applications to deal with wave breaking of the other  nonlinear dispersive model wave equations with higher nonlocal nonlinearities.

The outline of the paper is as follows. In Section 2, integrability of Eq.\eqref{m-m-mu-ch} is illuminated.
Section 3 is devoted to some basic facts and results, which will be used in the subsequent
sections. In Section 4, existence of the single peakon and multi-peakons of Eq.\eqref{m-m-mu-ch} are demonstrated. Local well-posedness in the Sobolev space to the Cauchy problem associated with Eq.\eqref{m-m-mu-ch} is then obtained in Section 5.  Blow-up criterion for strong solutions is established in Section 6. Finally in Section 7, with a new blow-up mechanism, sufficient conditions for wave breaking of strong solutions in finite time are described in detail.

%%%%%%%%%%%%%%%%%%%%%%%%%%%%%%%%%%%%%%%%%%%%%%%%%%%%%%%%%%%
%%%%%%%%%%%%%%%%%%%%%%%%%%%%%%%%%%%%%%%%%%%%%%%%%%%%%%%%%%%
\renewcommand{\theequation}{\thesection.\arabic{equation}}
\setcounter{equation}{0}

\section{Integrability and conservation laws}
In this section, we study integrability of
Eq.\eqref{m-m-mu-ch}. On one hand, Eq.\eqref{m-m-mu-ch} admits a bi-Hamiltonian structure, and it is formally integrable. Indeed, it can be written in the bi-Hamiltonian form
\begin{equation*}
m_t=J\frac{\delta H_1}{\delta m}=K\frac{\delta H_2}{\delta m},
 \end{equation*}
where
\begin{equation*}
J=-k_1\,\partial_x\, m\,\partial_x^{-1}\,m\,\partial_x -k_2(m\partial_x+\partial_x m)-\frac 12 \gamma \partial_x, \qquad {\rm and}
\qquad K=-\partial A=\partial_x^3,
\end{equation*}
are compatible Hamiltonian operators, while
\begin{equation}\label{cons-01}\begin{array}{c}
H_1=\frac 12 \int_{{\Bbb S}}mu \,dx, \qquad \text{and}
\\[0.3cm]
H_2=\int_{{\Bbb S}} \left
(\mu^2(u)u^2+\mu(u)u u^2_x-\frac 1{12}u^4_x \right) \, dx+k_2\int_{{\Bbb S}}(\mu(u)u+\frac 12 uu_x^2)dx
\end{array}
\end{equation}
are the corresponding Hamiltonian functionals. According to  Magri's
theorem  \cite{olv}, the associated recursion operator $\cR
= J\cdot K^{-1}$ produces a hierarchy of commuting Hamiltonian
functionals and  bi-Hamiltonian flows.

On the other hand, Eq.\eqref{m-m-mu-ch} admits the following Lax-pair
\begin{align*}
\begin{pmatrix} \psi_1\\\psi_2\end{pmatrix}_x=U(m)\begin{pmatrix}
\psi_1\\\psi_2\end{pmatrix}, \qquad \quad\begin{pmatrix}
\psi_1\\\psi_2\end{pmatrix}_t=V(m,u,\lambda)\begin{pmatrix}
\psi_1\\\psi_2\end{pmatrix},
\end{align*}
where $U$ and $V$ are given by
\begin{align*}
U(m,\lambda)=\begin{pmatrix}0 & \lambda \,m\\[0.2cm]-k_1 \lambda \,m-k_2 \lambda &
0\end{pmatrix},
\end{align*}
and
\begin{align*}
V(m,u,\lambda)=\begin{pmatrix}-\frac 12 u_x
& -\frac {\mu(u)}{2\lambda} -k_1 \lambda (2\mu(u)u-u^2_x)m-k_2 \lambda um \\[.2cm]
G & \frac 12 k_2 u_x\end{pmatrix},
\end{align*}
for $\gamma=0$ with
\begin{eqnarray*}
\begin{aligned}
G=&\;k_2\big(\frac 1{2\lambda}+\lambda k_2u\big)+\frac 1{2\lambda}k_1 \mu(u)+k_1^2\lambda (2u\mu(u)-u_x^2)m\\
&+k_1k_2\lambda \big(2u\mu(u)-u_x^2-um\big),
\end{aligned}
\end{eqnarray*}
and
\begin{align*}
U(m,\lambda)=\begin{pmatrix}-\sqrt{-\frac 12 \gamma}\lambda  & \lambda \,m\\
-k_1 \lambda \,m-k_2 \lambda &
\sqrt{-\frac 12 \gamma}\lambda\end{pmatrix},
\end{align*}
and
\begin{align*}
V(m,u,\lambda)=\begin{pmatrix}A
& B \\
G & -A\end{pmatrix},
\end{align*}
for $\gamma\not=0$ with \begin{eqnarray*}
&&A=\frac 14 \sqrt{-2\gamma}\big(\lambda^{-1}+2\lambda k_1(2u\mu(u)-u_x^2)+2\lambda k_2 u\big)-\frac 12 k_2 u_x,\\
&&B= -\frac {\mu}{2\lambda}+\frac 12 \sqrt{-2\gamma}u_x-k_1\lambda (2u\mu(u)-u_x^2)m-\lambda k_2 um.
\end{eqnarray*}

Let us now consider a scale limit of Eq.\eqref{m-m-mu-ch}. Applying the scaling transformation
$$\qeq{x\longmapsto \epsilon\, x,\\t\longmapsto \epsilon^{-1} t,\\ u\longmapsto \epsilon^2 u}$$
to Eq.\eqref{m-m-mu-ch} produces
\begin{eqnarray}\label{m-m-mu-ch-1}
\begin{aligned}
&(\epsilon^2 \mu(u)-u_{xx})_t+k_1 \big((2\epsilon^2 u\mu(u)-u_x^2)(\epsilon^2
\mu(u)-u_{xx})\big)_x\\
&+k_2\big(2u_x(\epsilon^2 \mu(u)-u_{xx})+u (\epsilon^2 \mu(u)-u_{xx})_x \big)+\gamma u_x=0.
\end{aligned}
\end{eqnarray}
Expanding
$$u(t,x)= u_0(t,x) + \epsilon\, u_1(t,x)  + \epsilon^2\, u_2(t,x) + \> \cdots $$
in powers of the small parameter $\epsilon$, the leading order term $u_0(t,x)$ satisfies
\begin{equation*}
-u_{0,xxt}+k_1 (u_{0,x}^2u_{0,xx})_x-k_2 \big(u_{0,x}u_{0,xx}+u_0u_{0,xxx} \big)+ \gamma u_{0,x}=0.
\end{equation*}
It follows that $v=u_{0,x}$ satisfies the following integrable equation
\begin{equation*}
v_{xt} -k_1 v_x^2v_{xx}+k_2 (vv_{xx}+\frac 12 v_x^2)-\gamma v=0,
\end{equation*}
which describes asymptotic dynamics of a short capillary-gravity wave \cite{fmn}, where $v(t,x)$ denotes the fluid velocity on the surface.

%%%%%%%%%%%%%%%%%%%%%%%%%%%%%%%%%%%%%%%%%%%%%%%%%%%%%%%%%%%%
\renewcommand{\theequation}{\thesection.\arabic{equation}}
\setcounter{equation}{0}
%%%%%%%%%%%%%%%%%%%%%%%%%%%%%%%%%%%%%%%%%%%%%%%%%%%%%%%%%%%%
\section{Preliminaries}

In this section, we are concerned with the following Cauchy problem
of Eq.\eqref{m-m-mu-ch} with $\gamma=0$
\begin{equation}\label{cauchy}
\left\{
 \begin{array}{ll}
\begin{split}
&m_t+k_1\big(2\mu(u)u-u^2_x)m\big)_x+k_2 (2mu_x+um_x)=0, \quad t>0,  \quad x \in \mathbb{R}, \\
&u(0,x)=u_0(x), \quad m=\mu(u)-u_{xx},\quad x \in \mathbb{R},\\
&u(t,x+1)=u(t,x),\quad t \ge 0, \quad x\in\mathbb{R}.
 \end{split}
\end{array} \right.
\end{equation}

In the following, for a given Banach space $Z$, we denote its norm
by $\|\cdot\|_Z$. Since all space of functions are over
$\mathbb{S}$, for simplicity, we drop $\mathbb{S}$ in our notations
of function spaces if there is no ambiguity. Various positive
constants will be uniformly denoted by common letters $c$ or $C$, which are
only allowed to depend on the initial data $u_0(x)$.

The notions of strong (or classical) and weak solutions will be
needed.

\begin{defi}\label{def-strong-solution}
If $u \in C([0, T]; H^{s}) \cap C^1([0, T]; H^{s-1})$ with
$s>\frac{5}{2}$ and some $T>0$ satisfies \eqref{cauchy}, then $u$ is
called a strong solution of \eqref{m-m-mu-ch} on $[0, T]$. If
$u(t,x)$ is a strong solution on $[0, T]$ for every $T >
0$, then it is called a global strong solution.
\end{defi}

Plugging $m$ in terms of $u$ into Eq.\eqref{m-m-mu-ch} results
in the following fully nonlinear partial differential equation:
\begin{eqnarray}\label{cauchy-1}
\begin{aligned}
&&u_t+ k_1\left [(2u\mu(u) -\fr1{3}u^2_x )u_x+\partial_x A^{-1}
(2\mu^2(u)u+\mu(u) u_x^2)+\fr{1}{3}\mu(u_x^3)\right]\\
&&+k_2\left[ uu_x+A^{-1}\partial_x(2u\mu(u)+\fr 12 u_x^2)  \right]
=0,
\end{aligned}
\end{eqnarray}
where $A=\mu-\partial_x^2$ is a linear operator, defined by $A(u)=\mu(u)-u_{xx}$, $A^{-1}$ is its inverse operator. Recall that \cite{klm}
\begin{equation}\label{g}
u=A^{-1}m= g\ast m,
\end{equation}
where $g$ is the Green function of the operator $A^{-1}$, given by
\begin{equation}\label{green}
g(x)=\fr{1}{2} (x-\fr 1{2})^2+\fr {23}{24}.
\end{equation}
Its derivative can be assigned to zero at $x=0$, so one has \cite{lmt}
\begin{eqnarray}\label{green-derivative}
g_x(x)\eqdefa\left\{\begin{array}{cc}0, &\; x=0,\\
x-\fr 1{2},&\; 0<x<1.
\end{array}\right.
\end{eqnarray}

The above formulation allows us to define a weak solution as
follows.

\begin{defi}\label{def-gws-3-1}
Given the initial data $u_{0} \in W^{1,3}$, a function $u \in
L^{\infty}([0, T), W^{1,3})$ is said to be a weak solution to the
Cauchy problem \eqref{cauchy}
%(or \eqref{cauchy-a})
if it satisfies the following identity:
\begin{eqnarray}\label{weaksol}
\begin{aligned}
&\int_0^{T}\int_{\mathbb{S}}\Big[u\,\varphi_{t}+k_1 \big(\mu(u)
u^2\varphi_{x} +\fr{1}{3}u_{x}^3\, \varphi -g_x\ast \left(2\mu^2(u)u
+\mu(u) u_x^2\right)\varphi\\&\qquad -\fr {1}{3}
\mu(u_x^3)\varphi \big)+k_2\big(\fr 12 u^2\varphi_x-g_x\ast (2u\mu(u)+\fr 12 u_x^2 )\varphi\big)\Big] \,dx\,dt\\&\qquad+\int_{\mathbb{S}}u_{0}(x)\,\varphi(0, x)\, dx=0,
\end{aligned}
\end{eqnarray}
for any smooth test function $\varphi(t,x)\in C^{\infty}_{c}([0,T
)\times {{\Bbb S}})$.  If $u$ is a weak solution on $[0, T)$ for
every $T
>0$, then it is called a global weak solution.
% to \eqref{cauchy} (or \eqref{cauchy-a}).
\end{defi}

The local well-posedness to the Cauchy problem \eqref{cauchy} will be established in the next section.

Denote
\begin{equation}\label{mu-1}
\mu_0=\mu(u(t))=\int_{\Bbb S}u(t,x)\ dx,\quad\mu_1=\left(\int_{\Bbb
S}u^2_x(t,x) d x\right)^{\frac1{2}}.
\end{equation}
It is easy to see from equation \eqref{m-mu-ch} and the conservation
law $H_1$ that $\mu(u)$ and $\mu_1$ are conserved. Furthermore, we
have the following result.

\begin{lemma}\label{sobolev-3}\cite{flq}
Let $u_0\in H^s$, $s>5/2$, and $u\in C([0,T); H^s)\cap C^1([0,T);
H^{s-1})$ is a solution of the Cauchy problem \eqref{cauchy}.
Then the following inequality holds:
\begin{equation}\label{inequality-2}
\|u(t,\cdot)-\mu(u(t,\cdot)) \|_{L^{\infty}}\le \frac{\sqrt{3}}{6}\mu_1.
\end{equation}
\end{lemma}

\begin{lemma}\label{l2.1}\cite{kat}
Let $r,t$ be real numbers such that $-r<t\le r$. Then
$$\|fg\|_{H^t}\le c\|f\|_{H^r}\|g\|_{H^t},\quad if\quad r>1/2,$$
$$\|fg\|_{H^{r+t-\frac1{2}}}\le c\|f\|_{H^r}\|g\|_{H^t},\quad if\quad r<1/2,$$
where $c$ is a positive constant depending on $r$ and $t$.
\end{lemma}

\begin{lemma}\label{l2.3}\cite{bcd,dan}
If $r>0$, then $H^r\cap L^{\infty}$ is a Banach algebra. Moreover
$$\|fg\|_{H^r}\le c(\|f\|_{L^{\infty}}\|g\|_{H^r}+\|f\|_{H^r}\|g\|_{L^{\infty}}),$$
where $c$ is a constant depending only on $r$.
\end{lemma}
\begin{lemma}\label{l2.4}\cite{bcd,dan}
If $r>0$, then
$$\|[\Lambda^r,f]g\|_{L^2}\le c(\|\partial_x f\|_{L^{\infty}}\|\Lambda^{r-1}g\|_{L^2}
+\|\Lambda^r f\|_{L^2}\|g\|_{L^{\infty}}),$$ where $c$ is a constant
depending only on $r$.
\end{lemma}
The following estimates for solutions to the one-dimensional transport equation have been used in \cite{dan,gl}.

\begin{lemma}\cite{bcd,dan}\label{lemtrans}
Consider the one-dimensional linear
transport equation
\begin{equation}\label{transport}
\partial_{t}f + v \,\partial_x f= g,\quad   f|_{t=0}=f_0.
\end{equation}
Let $0 \leq \sigma <1$, and suppose  that
\begin{align*}
f_{0}\in H^{\sigma}, &\quad
 g \in L^1([0, T]; H^{\sigma}),\\
 v_x \in L^1([0, T]; L^{\infty}),&\quad
 f \in L^{\infty}([0, T]; H^{\sigma})\cap C([0, T];
\mathcal{S}^{\prime}).
\end{align*}
Then $f \in C([0, T];H^{\sigma})$. More precisely, there
exists a constant $C$ depending only on $ \sigma$ such that, for every $0<t\leq T$,

\begin{equation*}
\begin{split}
\|f(t)\|_{H^{\sigma}}  \leq & e^{CV(t)}
\Big(\|f_0\|_{H^{\sigma}}+C\int_{0}^{t}\|g(\tau)\|_{H^{\sigma}}d\tau\Big)
\quad\hbox{\it with} \quad V(t)= \int_{0}^{t} \|\partial_x v(\tau)\|_{L^{\infty}}\,d\tau.
\end{split}
\end{equation*}
\end{lemma}

 As for the periodic case, thanks to Lemma 3.9 in \cite{cll}, we know that for every $u\in H^1({\mathbb S})$
\begin{align*}
\|u\|^2_{\mu}\leq\|u\|^2_{H^1}\leq3\|u\|^2_{\mu}
\end{align*}
where
$$\|u\|^2_{\mu}=\left. ((\mu-\partial^2_x)u,u\right. )_{L^2}
=\mu^2(u)+\int_{{\Bbb S}}u^2_x\ dx,$$ which implies that the norm $\|\cdot\|_{\mu}$ is equivalent to the norm $\|\cdot\|_{H^1}$.

%%%%%%%%%%%%%%%%%%%%%%%%%%%%%%%%%%%%%%%%%%%%%%%%%%%%%%%%%%%%
\renewcommand{\theequation}{\thesection.\arabic{equation}}
\setcounter{equation}{0}
%%%%%%%%%%%%%%%%%%%%%%%%%%%%%%%%%%%%%%%%%%%%%%%%%%%%%%%%%%%%

\section{Peaked solutions}

It is known that a remarkable property to the CH and $\mu$-CH equations is the existence of single and multi-peakons. In this section, we are concerned with the existence of single and multi-peakons to Eq.\eqref{m-m-mu-ch}. Recall
first that the single peakon and multi-peakons for $\mu$-CH equation and modified $\mu$-CH equation. Their single peakons are given respectively by \cite{klm}
\begin{eqnarray}\label{mu-ch-peakon-1}
u(t, x) =\frac{12}{13} c g(x-ct),
\end{eqnarray}
and \cite{qfl}
\begin{eqnarray}\label{m-mu-ch-peakon-1}
u(t, x) =\frac{2\sqrt{3c}}{5} g(x-ct),
\end{eqnarray}
where
\begin{eqnarray*}
g(x)=\frac 12 (x-[x]-\frac 12)^2+\frac {23}{24}
\end{eqnarray*}
with $[x]$ denoting the largest integer part of $x$.
Their multi-peakons are given by
\begin{eqnarray}\label{mu-ch-peakon-2}
u(t, x) = \sum\limits_{i=1}^N p_i(t)g(x-q^i(t)),
\end{eqnarray}
where $p_i(t)$ and $q^i(t)$ satisfy the following ODE system respectively for the $\mu$-CH equation \cite{lmt}
\begin{eqnarray*}
\begin{aligned}
&\dot{p}_i(t) =-\sum\limits_{j=1}^Np_ip_j g_x(q^i-q^j),\\
&\dot{q}^i(t) =\sum\limits_{j=1}^Np_jg(q^i(t)-q^j(t)),\;\; i=1,2,\cdots, N,\\
\end{aligned}
\end{eqnarray*}
and for the modified $\mu$-CH equation \eqref{m-m-mu-ch} \cite{qfl}
\begin{eqnarray*}
\begin{aligned}
\dot{p}_i(t) =&0,\\
\dot{q}^i(t) =&\frac 1{12} \Big[\sum\limits_{j,k\not=i}^N(p_j+p_k)^2+25 p_i^2\Big]+p_i\Big[\sum\limits_{j\not=i}^N p_j\big((q^i-q^j+\frac 12 \lambda_{ij})^2+\frac {49}{12}\big)\Big]\\
&+\sum\limits_{j<k, j,k\not=i}^N p_jp_k(q^j-q^k+\epsilon_{jk})^2,
\end{aligned}
\end{eqnarray*}
where $g_x$ ia defined by \eqref{green-derivative}, and
\begin{eqnarray}\label{lam-eps}
\begin{aligned}
&\lambda_{ij}=\left\{\begin{array}{c}1,\quad i<j\\
-1,\quad i>j,
\end{array}\right.\\
&\epsilon_{jk}=\left\{\begin{array}{c}1,\quad k-j\geq 2\\
0,\quad k-j\leq 1,
\end{array}\right.
\end{aligned}
\end{eqnarray}

Recall that the single peakon of the CH equation \eqref{ch} and the modified CH equation \eqref{m-ch} are given respectively by \cite{ch}
\begin{eqnarray*}
u=ce^{-|x-ct|}
\end{eqnarray*}
and \cite{gloq}
\begin{eqnarray*}
u=\sqrt{\frac{3c}{2}}e^{-|x-ct|}.
\end{eqnarray*}
Their periodic peakons are given respectively by \cite{con2,qll}
\begin{eqnarray*}
u=\frac{c}{\cosh(\frac 12)}\cosh \big( \frac 12 -(x-ct)+[x-ct]\big)
\end{eqnarray*}
and \cite{qll}
\begin{eqnarray*}
u=\sqrt{\frac {3c}{1+2\cosh^2(\frac 12 )}}\cosh \big( \frac 12 -(x-ct)+[x-ct]\big).
\end{eqnarray*}
Their multi-peakons are given by \cite{ch,cht}
\begin{eqnarray*}
u=\sum\limits_{i=1}^Np_i(t)e^{-|x-q^i(t)|},
\end{eqnarray*}
where $p_i(t)$ and $q^i(t)$ satisfy the system respectively for the CH equation \eqref{ch} with $\gamma=0$ \cite{cht}
\begin{eqnarray*}
\begin{aligned}
&\dot{p}_i=\sum\limits_{j=1}^Np_ip_j {\rm sign}(q^i-q^j) \;e^{-|q^i-q^j|},\\
&\dot{q}^i=\sum\limits_{j=1}^Np_je^{-|q^i-q^j|},
\end{aligned}
\end{eqnarray*}
and the system for the modified CH equation \eqref{m-ch} \cite{gloq}
\begin{eqnarray*}
\begin{aligned}
&\dot{p}_i=0,\\
&\dot{q}^i=\frac 23 p_i^2+2\sum\limits_{j= 1}^Np_ip_j e^{-|q^i-q^j|}+4\sum\limits_{1\leq j,j<i\leq N}p_k p_ie^{-|q^i-q^k|}.
\end{aligned}
\end{eqnarray*}

The existence of the single peakons of Eq.\eqref{m-m-mu-ch} is governed by the following result.
\begin{thm}\label{thm-peakon-1}
For any $c\geq -\fr {169k_2^2}{1200k_1}$, Eq.\eqref{m-m-mu-ch} admits
the peaked periodic-one traveling wave solution $u_c=\phi_c(\xi)$,
$\xi=x-ct$, where $\phi_c(\xi)$ is given by
\begin{eqnarray}\label{peakon-1}
\phi_c(\xi)=a \big[\fr 12 (\xi-\fr 12)^2+\fr {23}{24} \big]
\end{eqnarray}
with
\begin{equation}\label{peakon-root}
a=\left\{\begin{array}{cc}\fr {-13k_2\pm \sqrt{169k_2^2+1200 ck_1}}{50k_1},&  \;\; k_1\not=0,\\[.3cm]
\fr {12c}{13k_2},& \;\; k_1=0,\;\; k_2\not=0
\end{array}\right.
\end{equation}
for $\xi\in [-1/2,1/2]$ and $\phi(\xi)$ is extended periodically to the
real line.
\end{thm}

\begin{proof} In view of \eqref{peakon-1}, we assume that
the periodic peakon of \eqref{m-m-mu-ch} is given by
\begin{equation}\label{peakon-2}
u_c(t, x)=a\left[\fr 1{2}(\xi-[\xi]-\fr{1}{2})^2+\fr {23}{24}\right],
\end{equation}
where $a$ is a constant.
According to Definition \ref{def-gws-3-1}, $u_c$ satisfies the
equation
\begin{equation}
\label{peakon-3}
\begin{split}
\sum\limits_{j=1}^6 I_j:=&\int_0^T\int_{{\Bbb S}} u_{c,t} \varphi
dxdt+k_1 \int_0^T\int_{{\Bbb S}}
\big(2\mu(u_c)u_cu_{c,x}-\fr{1}{3}u_{c,x}^3\big)\varphi dxdt\\
&+k_1\int_0^T\int_{{\Bbb S}} g_x\ast \left[\mu
(u_c)\big(2\mu(u_c)u_c+u_{c,x}^2\big)\right]\varphi dxdt
+\fr{1}{3}k_1\int_0^T\int_{{\Bbb S}} \mu(u_{c,x}^3)\varphi dxdt\\
&+k_2\int_0^T\int_{{\Bbb S}} g_x\ast \big(2\mu(u_c)u_c+\frac 12 u_{c,x}^2\big)\varphi dxdt
+k_2\int_0^T\int_{{\Bbb S}} u_c u_{c,x}\varphi dxdt=0,
\end{split}
\end{equation}
for some $T>0$ and every test function $ \varphi(t, x)\in
C^{\infty}_{c} ([0,T)\times {\Bbb S})$. For any $x\in {{\Bbb S}}$, one
finds that
\begin{align*}
\mu(u_c)=&\;a\int_0^{ct}\left[\fr{1}{2}(x-ct+\fr{1}{2})^2+\fr{23}{24}\right]dx
\;+a\int_{ct}^1\left[\fr{1}{2}(x-ct-\fr{1}{2})^2+\fr{23}{24}\right]dx
=a,\\
\mu(u_{c,x}^3)=&\;a^3\int_0^{ct}(x-ct+\fr{1}{2})^3dx+a^3\int_{ct}^1(x-ct-\fr{1}{2})^3dx
=0.
\end{align*}

To compute $I_j$, $j=1,\ldots,6$, we need to consider two possibilities:
(i) $x>ct$, and (ii) $x \leq ct$. For $x>ct$, we have
$$2\mu(u_c)u_c-\fr 1{3}u_{c,x}^2=\fr{2}{3}a^2(\xi-\fr
1{2})^2+\fr{23}{12}a^2.
$$ On the other hand,
\begin{equation*}
\begin{split}
g_x\ast &\big(2\mu(u_c)u_c+u_{c,x}^2\big)\\
=&\;\int_{{\Bbb S}}\left(x-y-[x-y]-\fr 1{2}\right)\left(2a^2\big(y-ct-[y-ct]-\fr1{2}\big)^2+\fr{23}{12}a^2\right)dy\\
=&\;\int_0^{ct}\left(x-y-\fr 1{2}\right)\left(2a^2(y-ct+\fr1{2})^2+\fr{23}{12}a^2\right)dy\\
&\;+\int_{ct}^x\left(x-y-\fr 1{2}\right)\left(2a^2(y-ct-\fr1{2})^2+\fr{23}{12}a^2\right)dy\\
&\;+\int_{x}^1\left(x-y+\fr 1{2}\right)\left(2a^2(y-ct-\fr1{2})^2+\fr{23}{12}a^2\right)dy=\frac {1}{3}a^2 \xi(\xi-1)(1-2\xi),
\end{split}
\end{equation*}
\begin{equation*}
\begin{split}
g_x\ast & \big(2\mu(u_c)u_c+\frac 12 u_{c,x}^2\big)\\
=&\;\int_{{\Bbb S}}\left(x-y-[x-y]-\fr 1{2}\right)\left(\fr 32 a^2\big(y-ct-[y-ct]-\fr 1{2}\big)^2+\fr{23}{12}a^2\right)dy\\
=&\;\int_0^{ct}\left(x-y-\fr 1{2}\right)\left(\fr 32a^2(y-ct+\fr1{2})^2+\fr{23}{12}a^2\right)dy\\
&\;+\int_{ct}^x\left(x-y-\fr 1{2}\right)\left( \fr 32 a^2(y-ct-\fr1{2})^2+\fr{23}{12}a^2\right)dy\\
&\;+\int_{x}^1\left(x-y+\fr 1{2}\right)\left(\fr 32 a^2(y-ct-\fr1{2})^2+\fr{23}{12}a^2\right)dy=\frac {1}{4}a^2 \xi(\xi-1)(1-2\xi),
\end{split}
\end{equation*}
It follows that
\begin{equation*}
\begin{split}
I_1=&\; \int_0^T\int_{{\Bbb S}} u_{c,t} \varphi dxdt
=\;-ca\int_0^T\int_{{\Bbb S}} (\xi-\fr 1{2})\varphi(x,t)dxdt,\\
I_2=&\; k_1a^3\int_0^T\int_{{\Bbb S}}\left[\fr 2{3}(\xi-\fr 1{2})^3+\fr {23}{12}(\xi-\fr{1}{2})\right]\varphi(x,t) dxdt,\\
I_3=&\;\fr{1}{3}k_1a^3\int_0^T\int_{{\Bbb
S}}(-2\xi^3+3\xi^2-\xi)\varphi(x,t) dxdt,\;\; I_4=0,\\
I_5=&\frac 14 k_2 a^2\;\int_0^T\int_{{\Bbb
S}}(-2\xi^3+3\xi^2-\xi)\varphi(x,t) dxdt,\\
I_6=& k_2 a^2\;\int_0^T\int_{{\Bbb S}}\left[\frac 12 (\xi-\fr 12)^3+\fr {23}{24}(\xi-\fr 12)\right]\varphi(x,t) dxdt.
\end{split}
\end{equation*}
Plugging above expressions into \eqref{peakon-3}, we deduce that
for any $\varphi(x,t)\in C_c^{\infty}([0,T)\times {\Bbb S})$
\begin{equation*}
\sum\limits_{j=1}^6 I_j=\int_0^T\int_{{\Bbb
S}}a \left(\fr{25}{12}k_1a^2+\fr {13}{12}k_2 a -c\right )(\xi-\fr1{2})\varphi(x,t)dxdt.
\end{equation*}
A similar computation yields for $x\leq ct$ that
\begin{eqnarray*}
&2\mu(u_c)u_c-\fr 1{3}u_{c,x}^2=\;\fr{2}{3}a^2(\xi+\fr 1{2})^2+\fr{23}{12}a^2,\\
&2\mu(u_c)u_c+\fr 1{2}u_{c,x}^2=\;\fr{3}{2}a^2(\xi+\fr 1{2})^2+\fr{23}{12}a^2,
\end{eqnarray*}
and
\begin{eqnarray*}
\begin{aligned}
&g_x\ast \big(2\mu(u_c)u_c+u_{c,x}^2\big)
=\;\frac {1}{3}a^2 \xi(\xi-1)(1-2\xi),\\
&g_x\ast \big(2\mu(u_c)u_c+\frac 12 u_{c,x}^2\big)
=\;\frac {1}{4}a^2 \xi(\xi+1)(1+2\xi),\\
\end{aligned}
\end{eqnarray*}
This allows us to compute
\begin{eqnarray*}
\begin{aligned}
\sum\limits_{j=1}^4I_j=&\;\int_0^T\int_{{\Bbb
S}}(\fr {25}{12}k_1 a^3-ca)(\xi-\fr 12 )\varphi(x,t) dxd,\\
I_5=&\;\fr 14 k_2 a^2 \int_0^T\int_{\Bbb S}(2\xi^3+3\xi^2 +\xi)\varphi(x,t)dxdt,\\
I_6=&\;k_2 a^2 \int_0^T\int_{\Bbb S}\left[\fr 12 (\xi+\fr 12)^2+\fr {23}{24}(\xi+\fr 12 )\right]\varphi(x,t)dxdt,\\
\end{aligned}
\end{eqnarray*}
whence we arrive at
\begin{equation*}
\sum\limits_{j=1}^6 I_j=\int_0^T\int_{{\Bbb
S}}a \left(\fr{25}{12}k_1a^2+\fr {13}{12}k_2 a -c\right)(\xi+\fr1{2})\varphi(x,t)dxdt.
\end{equation*}
Since $\varphi$ is arbitrary, both cases imply that $a$ fulfills the equation
$$\fr{25}{12}k_1a^2+\fr {13}{12}k_2 a-c=0. $$
Clearly, its solutions are given by \eqref{peakon-root}. Thus the theorem is proved.
\end{proof}

%\begin{rmk}
%It is easy to see from Theorem \ref{thm-peakon-1} that Eq. \eqref{m-m-mu-ch} admits peakons only when $c\geq -\fr {169k_2^2}{1200k_1}$.
%\end{rmk}

Furthermore, one can show that Eq.\eqref{m-m-mu-ch} admits
the multi-peakons of the form \eqref{mu-ch-peakon-2},
where $p_i(t)$ and $q^i(t)$, $i=1,2,\ldots,N$, satisfy the following ODE system
\begin{eqnarray}\label{peakon-4}
\begin{split}
&\dot{p}_i+k_2\sum\limits_{j=1}^Np_ip_j(q^i-q^j-\fr 12 )=0,\\
&\dot{q}_i-k_1\big[ \frac 1{12}\big(23 \sum\limits_{j,k\not=i}(p_j+p_k)^2+25 p_i^2\big)\\
&-p_i \big( \sum\limits_{j\not= i}p_j(q^i-q^j)^2+\frac 12 \lambda_{ij})^2+\frac {49}{12} \big) \\
&-\sum\limits_{j<k, j,k\not=i}p_jp_k(q^j-q^k+\epsilon_{jk} )^2 \big]\\
&-k_2\sum\limits_{j=1}^N p_j\big(\frac 12 (q^i-q^j)^2-\frac 12 |q^i-q^j|+\frac {13}{12}\big)=0,
\end{split}
\end{eqnarray}
where $\lambda_{ij}$ and $\epsilon_{jk}$ are given by \eqref{lam-eps}.

In particular, when $N=2$, system \eqref{peakon-4} can be solved explicitly, which yields
\begin{eqnarray}\label{two-peakon-1}
\begin{aligned}
p_1=&\,\frac {ae^{b(t-t_0)}}{1+e^{b(t-t_0)}},\;\; p_2= \frac {a}{1+e^{b(t-t_0)}},\\
q^1=&\;-\frac {k_1 a^2}{b} \left(\fr 1{12}+(\fr 12 -a_1)^2 \right) \frac 1{1+e^{b(t-t_0)}}\\
&\;+\frac a {12}\big( 23 k_1 a +6 a_1 (a_1-1)k_2+13k_2 \big) (t-t_0)\\
&+\frac a{6b}\big( k_1 a-3a_1(a_1-1)k_2 \big)\ln (1+e^{b(t-t_0)})+c_1,\\
q^2=&\;q^1+a,
\end{aligned}
\end{eqnarray}
where $a$, $a_1$, $b>0$ and $t_0$ are some constants.

%%%%%%%%%%%%%%%%%%%%%%%%%%%%%%%%%%%%%%%%%%%%%%%%%%%%%%%%%%%%
\renewcommand{\theequation}{\thesection.\arabic{equation}}
\setcounter{equation}{0}
%%%%%%%%%%%%%%%%%%%%%%%%%%%%%%%%%%%%%%%%%%%%%%%%%%%%%%%%%%%%
\section{The local well-posedness}
In this section, we shall discuss the local
well-posedness of the Cauchy problem (\ref{cauchy}). Our result is the following theorem:
\begin{theorem}\label{t3-1}
Suppose that $u_0\in H^s(\mathbb{S})$ for some constant $s>5/2$. Then there
exists $T>0$, which depends only on $\|u_0\|_{H^s}$, such that  problem (\ref{cauchy}) has
a unique solution $u(t,x)$ in the space $C([0,T);H^s(\mathbb{S}))\bigcap C^1([0,T);H^{s-1}(\mathbb{S}))$. Moreover,
the solution $u$ depends continuously on the initial data $u_0$ in
the sense that the mapping of the initial data to the solution is
continuous from the Sobolev space $H^s$ to the space $C([0,T);H^s(\mathbb{S}))\bigcap C^1([0,T);H^{s-1}(\mathbb{S}))$.
\end{theorem}
In the remaining part of this section, we will concentrate on proving this theorem. The proof is approached by the viscosity method, which has been applied to establish the local well-posedness for the KdV equation \cite{bon} and the CH equation \cite{lo}.

Firstly, we consider the Cauchy problem for a regularized version of the first
equation of (\ref{cauchy}).
\begin{eqnarray}\label{e3.1}
\left\{ \begin{aligned}
&m_t-\varepsilon m_{xxt}=-k_1((2\mu(u)u-u^2_x)m)_x-k_2(2mu_x+um_x), \quad t>0,  \quad x \in \mathbb{R}, \\
&u(t,x+1)=u(t,x), \quad t>0,  \quad x \in \mathbb{R}, \\
&u(0,x)=u_0(x)\in H^s, \quad x \in \mathbb{R},\quad s\ge1.
\end{aligned} \right.
\end{eqnarray}
Here $\varepsilon$ is a constant and $0<\varepsilon<1/16$. We start by inverting the linear
differential operator on the left hand side.
For any $0<\varepsilon<1/16$ and any $s$, the integral operator
\begin{eqnarray}\label{e3.2}
\mathcal{D}=(1-\varepsilon\partial^2_x)^{-1}(\mu-\partial^2_x)^{-1}:H^s\rightarrow
H^{s+4}
\end{eqnarray}
defines a bounded linear operator on the indicated Sobolev spaces.

To prove the existence of a solution to the problem in (\ref{e3.1}), we
use the operator (\ref{e3.2}) to both sides of the equation in
(\ref{e3.1}) and then integrate the resulting equation with respect
to time $t$. This leads to the following equation
\begin{align*}
u(t,x)=u_0(x)-\int^t_0\mathcal{D}&\left\{k_1\left[2\mu^2(u)u_x-\mu(u)\partial^3_x(u^2)+\mu(u)\partial_x(u^2_x)
+\frac1{3}\partial_x(u^3_x)\right]\right.\nonumber\\&\left.+k_2\left[2\mu(u)u_x-\frac1{2}\partial^3_x(u^2)
+\frac1{2}\partial_x(u_x^2)
 \right]\right\}(\tau,x)\ d\tau.
\end{align*}
A standard application of the contraction mapping theorem leads to
the following existence result which is stated without the detailed proof.
\begin{theorem}\label{t3.1}
For each initial data $u_0\in H^s$ with $s\ge1$, there exists a
$T>0$ depending only on the norm of $u_0$ in $H^s$ such that there
corresponds a unique solution $u(t,x)\in C([0,T];H^s)$ of the
Cauchy problem (\ref{e3.1}) in the sense of distribution. If $s\ge2$,
the solution $u(t,x)\in C^{\infty}([0,\infty);H^s)$ exists for all time. In particular, when
$s\ge4$, the corresponding solution is a classical globally defined
solution of (\ref{e3.1}).
\end{theorem}

The global existence result follows from the conservation law
$$\mu^2(u)(t)+\int_{\mathbb{R}}(u_x^2+\varepsilon u^2_{xx})\ dx=
\mu^2(u)(0)+\int_{\mathbb{R}}(u_{0x}^2+\varepsilon
u^2_{0xx})\ dx,$$ admitted by (\ref{e3.1}) in its integral form.

Now we establish the a priori estimates for solutions of
(\ref{e3.1}) using energy estimates.
\begin{theorem}\label{t3.2}
Suppose that for some $s\ge4$, the function $u(t,x)$ is the solution
of (\ref{e3.1}) corresponding to the initial data
$u_0\in H^s$. For any real number $q\in(1/2,s-2]$, there exists a constant $c$
depending only on $q$ such that
\begin{align}\label{e3.3}
&\int_{\mathbb{S}}(\Lambda^{q+2}u)^2\
dx\le\int_{\mathbb{S}}\big[(\Lambda^{q+2}u_0)^2+\varepsilon(\Lambda^{q+3}u_0)^2\big]\
dx+c\int^t_0\left(\int_{\mathbb{S}}(\Lambda^{q+2}u)^2\,dx+1\right)^2\
d\tau.
\end{align}
For any $q\in[0,s-2]$, there exists a constant $c$ such that
\begin{align}\label{e3.4}
\|u_t\|_{H^{q+1}}\le c\|u\|^2_{H^{q+2}} \left(1+\|u\|_{H^{q+2}}\right).
\end{align}
\end{theorem}
\begin{proof}
First of all, we recall that the operator $\Lambda=(1-\partial^2_x)^{1/2}$ and $\|v\|_{H^s}=\|\Lambda^s v\|_{L^2}$.
Note that $\mu(u)$ is free of $x$, then the equation in \eqref{e3.1} is equivalent to the following equation:
\begin{align*}
\mu_t(u)&-u_{xxt}+\varepsilon u_{xxxxt}\\=&-k_1\left(2\mu^2(u)u_x-4\mu(u)u_xu_{xx}+2u_xu^2_{xx}+u^2_xu_{xxx}-2\mu(u)uu_{xxx}\right)\\
&-k_2\big(2\mu(u)u_x-2u_xu_{xx}-uu_{xxx}\big).
\end{align*}
For any $q\in(1/2,s-2]$, applying $(\Lambda^qu)\Lambda^q$ to both sides of the above
equation and integrating with respect to $x$, we have the following estimate
\begin{align*}
\frac1{2}\dfrac{d}{dt}&\big(\mu^2(\Lambda^qu)+\|\Lambda^qu_x\|^2_{L^2}+\varepsilon\|u_{xx}\|^2_{H^q}\big)\\
=&(2k_1\mu(u)+k_2)\int_{\Bbb S}\left(\dfrac{3}{2}\Lambda^qu_x\Lambda^qu^2_{x}+\Lambda^qu_{xx}\Lambda^q(uu_{x})
-\Lambda^qu_x\Lambda^qu^2_x\right)\,dx\\
&-\dfrac{k_1}{3}\int_{\Bbb S}\Lambda^qu_{xx}\Lambda^qu^3_{x}\,dx.
\end{align*}
Using the Cauchy-Schwartz inequality, $|\mu(u)|\le c\|u\|_{L^2}\le c\|u\|_{H^{q+2}}$, we deduce that
\begin{align}\label{e3.5}
\frac1{2}\dfrac{d}{dt}\big(\mu^2(\Lambda^qu)+\|\Lambda^qu_x\|^2_{L^2}+\varepsilon\|u_{xx}\|^2_{H^q}\big)\le
c(\|u\|^4_{H^{q+2}}+\|u\|^3_{H^{q+2}}).
\end{align}
Next, we differentiate the equation in \eqref{e3.1} with regard to $x$ and then obtain:
\begin{align*}
-u_{xxxt}&+\varepsilon u_{xxxxxt}\\&=-k_1\left(2\mu^2(u)u_x-4\mu(u)u_xu_{xx}+2u_xu^2_{xx}+u^2_xu_{xxx}-2\mu(u)uu_{xxx}\right)_x\\
&\quad-k_2\big(2\mu(u)u_x-2u_xu_{xx}-uu_{xxx}\big)_x.
\end{align*}
Applying $\Lambda^qu_x\Lambda^q$ to the above equation and
integrating with respect to $x$ over $\Bbb S$ and noting that
$$\int_{\Bbb S}\Lambda^qf_x\Lambda^qg_x\,dx=-\int_{\Bbb S}\Lambda^{q+1}f\Lambda^{q+1}g\,dx+\int_{\Bbb S}\Lambda^qf\Lambda^qg\,dx,$$
we arrive at
 \begin{align}\label{e3.6}
\frac1{2}\dfrac{d}{dt}\big(\mu^2(\Lambda^qu_x)+\|\Lambda^qu_{xx}\|^2_{L^2}+\varepsilon\|u_{xxx}\|^2_{H^q}\big):\equiv
J_1+J_2+J_3,
\end{align}
where
\begin{align*}
J_1=&\;(4k_1\mu(u)+2k_2)\int_{\Bbb S}\Lambda^qu_x\Lambda^q(u_xu_{xx})_x\,dx\\
=&\;(2k_1\mu(u)+k_2)\int_{\Bbb S}\big(\Lambda^{q+1}u_x\Lambda^{q+1}u^2_x-\Lambda^qu_x\Lambda^qu^2_x\big)\,dx,\\
J_2=&\;(2k_1\mu(u)+k_2)\int_{\Bbb S}\Lambda^qu_x\Lambda^q(uu_{xxx})_x\,dx\\
=&\;(2k_1\mu(u)+k_2)\int_{\Bbb S}\Big(\Lambda^{q+1}u_x[\Lambda^{q+1},u]u_{xx}+u\Lambda^{q+1}u_{x}\Lambda^{q+1}u_{xx}
\\&\;\qquad\qquad\qquad\qquad-\Lambda^{q}u_{x}\Lambda^{q}(uu_{xx})
+\Lambda^{q}u_{xx}\Lambda^{q}(u_xu_{xx})\Big)\,dx,
\end{align*}
and
\begin{align*}
J_3=&-k_1\int_{\Bbb S}\Lambda^qu_x\Lambda^q(2u_xu^2_{xx}+u_x^2u_{xxx})_x\,dx=k_1\int_{\Bbb S}\Lambda^qu_{xx}\Lambda^q(u^2_xu_{xx})_x\,dx\\
=&-k_1\int_{\Bbb S}\left\{[\Lambda^{q+1},u^2_{x}]u_{xx}\Lambda^{q+1}u_x+(u_x)^2\Lambda^{q+1}u_{x}\Lambda^{q+1}u_{xx}-
\Lambda^qu_{x}\Lambda^q(u^2_xu_{xx})\right\}\,dx.
\end{align*}
Applying Lemma \ref{l2.4}, the Cauchy-Schwartz inequality and observing the fact that $H^q$ is
 a Banach algebra for $q>1/2$, we obtain
\begin{align*}
J_1\le c(1+\|u\|_{H^{q+2}})\|u\|^3_{H^{q+2}}\le c(\|u\|^4_{H^{q+2}}+\|u\|^3_{H^{q+2}}),
\end{align*}
and
\begin{align*}
J_2\le &c(1+\|u\|_{H^{q+2}})\left(\|u\|_{H^{q+2}}(\|u_x\|_{L^{\infty}}\|u_{xx}\|_{H^q}
+\|u\|_{H^{q+1}}\|u_{xx}\|_{L^{\infty}})
+3\|u\|^3_{H^{q+2}}\right),\\
\le &c(\|u\|^4_{H^{q+2}}+\|u\|^3_{H^{q+2}}),\\
J_3\le &c\left(\|u_x\|_{H^{q+1}}(2\|u_xu_{xx}\|_{L^{\infty}}\|u_{xx}\|_{H^{q}}
+\|u_x\|^2_{H^{q+1}}\|u_{xx}\|_{L^{\infty}})+3\|u\|^4_{H^{q+2}}\right)\\
\le &c(\|u\|^4_{H^{q+2}}+\|u\|^3_{H^{q+2}}).
\end{align*}
Collecting above expressions, we have
\begin{align}\label{e3.6}
\frac1{2}\dfrac{d}{dt}(\mu^2(\Lambda^qu_x)+\|\Lambda^qu_{xx}\|^2_{L^2}+\varepsilon\|u_{xxx}\|^2_{H^q})\le c(\|u\|^4_{H^{q+2}}+\|u\|^3_{H^{q+2}}).
\end{align}
Whence from \eqref{e3.5} and \eqref{e3.6}, one finds that
\begin{align*}
\dfrac{d}{dt}&(\mu^2(\Lambda^qu)+\mu^2(\Lambda^qu_x)+\|\Lambda^qu_x\|^2_{L^2}
+\|\Lambda^qu_{xx}\|^2_{L^2}+\varepsilon\|u_{xx}\|^2_{H^q}
+\varepsilon\|u_{xxx}\|^2_{H^q})\\&\le c(\|u\|^4_{H^{q+2}}+\|u\|^3_{H^{q+2}})
\le c(1+\|u\|^2_{H^{q+2}})^2.
\end{align*}
Integrating the above inequality from $0$ to $t$ yields
\begin{align*}
\mu^2&(\Lambda^qu)(t)+\mu^2(\Lambda^qu_x)(t)+\|\Lambda^qu_x(t)\|^2_{L^2}
+\|\Lambda^qu_{xx}(t)\|^2_{L^2}+\varepsilon\|u_{xx}(t)\|^2_{H^q}
+\varepsilon\|u_{xxx}(t)\|^2_{H^q}\\
\le& \mu^2(\Lambda^qu_0)+\mu^2(\Lambda^qu_{0,x})+\|\Lambda^qu_{0,x}\|^2_{L^2}
+\|\Lambda^qu_{0,xx}\|^2_{L^2}+\varepsilon\|u_{0,xx}\|^2_{H^q}
+\varepsilon\|u_{0,xxx}\|^2_{H^q}\\&+ c\int_0^t\left(1+\|u(\tau)\|^2_{H^{q+2}}\right)^2\,d\tau.
\end{align*}
In view of the equivalence of the norms $\|\cdot\|_{\mu}$ and $\|\cdot\|_{H^1}$, we get the result of \eqref{e3.3}.

In order to estimate the norm of $u_t$, applying the operator $\Lambda^qu_t\Lambda^q$ to the equation in \eqref{e3.1} and integrating the result with respect to $x$ lead to
 \begin{align*}
&\int_{\Bbb S}\Lambda^qu_t\Big(\Lambda^q(\mu-\partial^2_x)u_t+\varepsilon\Lambda^qu_{xxxxt}\Big)\,dx
=\mu^2(\Lambda^qu_t)+\|\Lambda^qu_{tx}\|^2_{L^2}+\varepsilon\|\Lambda^qu_{xxt}\|^2_{L^2},\\
&\left|(4k_1\mu(u)+2k_2)\int_{\Bbb S}\Lambda^qu_t\Lambda^q(u_xu_{xx})\,dx\right|
\le c(1+\|u\|_{H^{q+1}})\|u_t\|_{H^{q+1}}\|u\|^2_{H^{q+1}},\\
&\left|-2\mu(u)(k_1\mu(u)+k_2)\int_{\Bbb S}\Lambda^qu_t\Lambda^q(u_x)\,dx\right|
\le c(1+\|u\|_{H^{q+1}})\|u_t\|_{H^{q+1}}\|u\|_{H^{q+1}},\\
&\left|(2k_1\mu(u)+k_2)\int_{\Bbb S}\Lambda^qu_t\Lambda^q(uu_{xxx})\,dx\right|\\&\quad=\left|-(2k_1\mu(u)+k_2)\int_{\Bbb S}(\Lambda^qu_{tx}\Lambda^q(uu_{xx})+\Lambda^qu_t\Lambda^q(u_xu_{xx}))\,dx\right|\\
&\quad\le c\|u_{tx}\|_{H^{q}}\left(\|u\|^3_{H^{q+2}}+\|u\|^2_{H^{q+2}}\right),
\end{align*}
and
\begin{align*}
\left|-k_1\int_{\Bbb S}\Lambda^qu_t\Lambda^q(2u_xu^2_{xx}+u_x^2u_{xxx})\,dx\right|=\left|k_1\int_{\Bbb S}
\Lambda^qu_{tx}\Lambda^q(u^2_xu_{xx})\,dx\right|\le c\|u_{tx}\|_{H^{q}}\|u\|^3_{H^{q+2}}.
\end{align*}
Similarly, in view of the equivalence of the norms $\|\cdot\|_{\mu}$ and $\|\cdot\|_{H^1}$, it follows that
\begin{align*}
\|u_t\|^2_{H^{q+1}}\le\mu^2(\Lambda^qu_t)+\|\Lambda^qu_{tx}\|^2_{L^2}+\varepsilon\|\Lambda^qu_{xxt}\|^2_{L^2}
\le c \|u_t\|_{H^{q+1}}\left(\|u\|^3_{H^{q+2}}+\|u\|^2_{H^{q+2}}\right).
\end{align*}
Applying the Cauchy inequality, we deduce that
\begin{align*}
\|u_t\|^2_{H^{q+1}}\le \dfrac1{2}\|u_t\|^2_{H^{q+1}}+c^2\left(\|u\|^3_{H^{q+2}}+\|u\|^2_{H^{q+2}}\right)^2.
\end{align*}
Thus, the above inequality implies that
\begin{align*}
\|u_t\|^2_{H^{q+1}}\le C\left(\|u\|^3_{H^{q+2}}+\|u\|^2_{H^{q+2}}\right)^2.
\end{align*}
This completes the proof of this theorem.
\end{proof}

To show the existence of the solutions to the initial value problem
(\ref{cauchy}), we regularize its initial data $u_0$.
For any fixed real number $s>5/2$, suppose that the function $u_0\in
H^s$, and let $u_0^{\varepsilon}$ be the convolution
$u_0^{\varepsilon}=\phi^{\varepsilon}\ast u_0$ of the functions
$\phi^{\varepsilon}(x)=\varepsilon^{-1/16}\phi(\varepsilon^{-1/16}x)$
such that the Fourier transform $\hat{\phi}$ of $\phi$ satisfies
$\hat{\phi}\in C^{\infty}_0,\hat{\phi}(\xi)\ge0$ and
$\hat{\phi}(\xi)=1$ for any $\xi\in(-1,1)$. Then it follows from
Theorem \ref{t3.2} that for each $\varepsilon$ with
$0<\varepsilon<1/16 $ the Cauchy problem
\begin{align*}
\left\{
 \begin{aligned}
&\mu_t(u)-u_{xxt}+\varepsilon u_{xxxxt}=-k_1\big((2\mu(u)u-u^2_x)m\big)_x-k_2(2mu_x+um_x), \,\, t>0,  \,\, x \in \mathbb{R}, \\
&u(0,x)=u_0^{\varepsilon}(x), \quad x \in \mathbb{R},
\end{aligned} \right.
\end{align*}
which is equivalent to the following problem:
\begin{align}\label{r3.1}
\left\{
 \begin{aligned}
&u_t-\varepsilon u_{xxt}=-\big(2k_1\mu(u)u+k_2u\big)u_x+\frac{k_1}{3}\big(u^3_x+\mu(u^3_x)\big)\\
&\qquad\qquad\qquad-\partial_x(\mu-\partial_x^2)^{-1}
\left[k_1\big(2\mu^2(u)u+\mu(u)u_x^2\big)+k_2\big(2\mu(u)u+\fr 1{2}u^2_x\big)\right], \\
&u(0,x)=u_0^{\varepsilon}(x), \quad x \in \mathbb{R},
\end{aligned} \right.
\end{align}
and has a unique solution $u^{\varepsilon}(x,t)\in
C^{\infty}([0,\infty);H^{\infty})$. To show that $u^{\varepsilon}$
converges to a solution of the problem (\ref{cauchy}), we first
demonstrate the properties of the initial data $u_0^{\varepsilon}$ in
the following theorem. The proof is similar to that of Lemma 5 in
\cite{bon}.
\begin{lemma}\label{t3.3}
Under the above assumptions, there hold
\begin{align}
&\|u_0^{\varepsilon}\|_{H^q}\le
c,&if\quad q\le s,\label{i3.1}\\
&\|u_0^{\varepsilon}\|_{H^q} \le c\varepsilon^{(s-q)/16}, &if\quad
q> s,\label{i3.2}\\
 &\|u_0^{\varepsilon}-u_0\|_{H^q} \le
c\varepsilon^{(s-q)/16}, &if\quad q\le s,\label{i3.3}\\
&\|u_0^{\varepsilon}-u_0\|_{H^s}=o(1),\label{i3.4}
\end{align}
for any $\varepsilon$ with $0<\varepsilon<1/16$, where $c$ is a
constant independent of $\varepsilon$.
\end{lemma}
Combining the estimates in Lemma \ref{t3.3} and the a priori estimates in Theorem \ref{t3.2}, we
shall evaluate norms of the function $u^{\varepsilon}$ in the
following theorem, which will be used to show the convergence of
$\{u^{\varepsilon}\}$.
\begin{lemma}\label{t3.4}
There exist constants $c_1,c_2$ and $M$ such that the following
inequalities hold for any $\varepsilon$ sufficiently small and
$t<1/(cM)$:
\begin{align*}
\|u^{\varepsilon}\|_{H^{s}}&\le\dfrac{c_1}{(1-cMt)^{c_2}},\\
\|u^{\varepsilon}\|_{H^{s+p}}&\le\dfrac{c_1\varepsilon^{-p/16}}{(1-cMt)^{c_2}},\quad p>0, \\
\|u^{\varepsilon}_t\|_{H^{s+p}}&\le\dfrac{c_4\varepsilon^{-3(p+1)/16}}{(1-cMt)^{c_3}},\quad
p>-2.
\end{align*}
\end{lemma}
\begin{proof}
Choose a fixed number $q=s-2$. It follows from
(\ref{e3.3}) that
\begin{align*}
&\int_{\mathbb{S}}(\Lambda^su^{\varepsilon})^2\
dx\le\int_{\mathbb{S}}\Big((\Lambda^su_0^{\varepsilon})^2+\varepsilon(\Lambda^su_{0,x}^{\varepsilon})^2\Big)\
dx+c\int^t_0 \Big(\int_{\mathbb{S}}(\Lambda^su^{\varepsilon})^2\
dx+1\Big)^2 \;d\tau.
\end{align*}
Then the following inequality
\begin{align*}
\|u^{\varepsilon}\|^2_{H^s}=\int_{\mathbb{S}}(\Lambda^su^{\varepsilon})^2\
dx\le \dfrac{M_s+1}{1-ct(M_s+1)}\le\dfrac{M}{1-cMt},
\end{align*}
 holds for any $t\in[0,1/(cM))$, where
$$M_s=\int_{\mathbb{S}}\Big((\Lambda^su_0^{\varepsilon})^2+\varepsilon(\Lambda^su_{0,x}^{\varepsilon})^2\Big)\
dx,\qquad\text{and}\qquad M=1+M_s.$$

Let $q=s+p-2$. In a similar way, applying Lemma \ref{t3.3} to \eqref{e3.3}, one may obtain the inequality
\begin{align*}
\|u^{\varepsilon}\|_{H^{s+p}}\le\dfrac{c_1(1+\varepsilon^{-p/8}
+\varepsilon^{1-(p+1)/8})^{\frac1{2}}}{(1-cMt)^{c_2}}
\le\dfrac{c_1\varepsilon^{-p/16}}{(1-cMt)^{c_2}},
\end{align*}
for some constant $c_1$. Then applying the above results and Lemma \ref{t3.3} to \eqref{e3.4} with $q=s+p-1$, we deduce that
\begin{align*}
\|u^{\varepsilon}_t\|_{H^{s+p}}\le\dfrac{c_4\varepsilon^{-3(p+1)/16}}{(1-cMt)^{c_3}},
\end{align*}
for some constant $c_4$. Thus this completes the proof of the theorem.
\end{proof}

We now show that $\{u^{\varepsilon}\}$ is a Cauchy sequence. Let
$u^{\varepsilon}$ and $u^{\delta}$ be solutions of (\ref{r3.1}),
corresponding to the parameters $\varepsilon$ and $\delta$,
respectively, with $0<\varepsilon<\delta<1/16$, and let
$w=u^{\varepsilon}-u^{\delta},\ f=u^{\varepsilon}+u^{\delta}$. Then
$w$ satisfies the following problem
\begin{equation}\label{m3.1}
\left\{
 \begin{aligned}
&w_t-\varepsilon w_{xxt}+(\delta-\varepsilon)u^{\delta}_{xxt}
\\&=-2k_1\mu(w)u^{\varepsilon}u^{\varepsilon}_x
-\Big(2k_1\mu(u^{\delta})+k_2\Big)\Big(wu^{\varepsilon}_x+u^{\delta}w_x\Big)\\&\quad
 +\frac{k_1}{3}w_x\Big((u^{\varepsilon}_x)^2+u_x^{\varepsilon}u_x^{\delta}+(u_x^{\delta})^2\Big)
 -\frac{k_1}{3}\mu\left(w_x\Big((u^{\varepsilon}_x)^2+u_x^{\varepsilon}u_x^{\delta}
 +(u_x^{\delta})^2\Big)\right)\\&\quad
-\partial_x(\mu-\partial_x^2)^{-1}
\bigg\{
k_1\Big[
2\mu(w)\mu(f)u^{\varepsilon}+2\mu^2(u^{\delta})w
+\mu(w)(u^{\varepsilon}_x)^2+\mu(u^{\delta})w_xf_x\Big]
\\&\qquad\qquad\qquad\qquad
+k_2\Big[2\mu(w)u^{\varepsilon}+2\mu(u^{\delta})w
+\frac1{2}w_xf_x
\Big]\bigg\}, \\
&w(0,x)=u_0^{\varepsilon}(x)-u_0^{\delta}(x), \quad x \in \mathbb{R}.
\end{aligned} \right.
\end{equation}
\begin{theorem}\label{t3.5}
There exists $T>0$, such that $\{u^{\varepsilon}\}$ is a Cauchy sequence
in the space $C([0,T);H^{s}), s > 5/2$.
\end{theorem}
\begin{proof}
For a constant $q$ with $ 1/2<q<\min\{1,s-2\}$, multiplying
$\Lambda^qw\Lambda^q$ to both sides of the equation in (\ref{m3.1}) and then
integrating with respect to $x$ over  ${\mathbb S}$, we obtain
\begin{align*}
&\int_{\Bbb S}\Lambda^qw\Big(
\Lambda^q(w_t-\varepsilon w_{xxt})
\Big)\,dx
=\frac1{2}\dfrac{d}{dt}\left(
\|\Lambda^qw\|^2_{L^2}+\varepsilon\|\Lambda^qw_{x}\|^2_{L^2}
\right),\\
&\left|\int_{\Bbb S}\Lambda^qw\Lambda^qu^{\delta}_{
xxt}\,dx\right|
=\left|\int_{\Bbb S}\Lambda^qw_{x}\Lambda^qu^{\delta}_{
xt}\,dx\right|
\le c\|u_t^{\delta}\|_{H^{q+1}}\|w\|_{H^{q+1}}.
\end{align*}
Similar to the proof of \eqref{e3.3}, for the right hand side of the equation in (\ref{m3.1}) we obtain
 \begin{align*}
  &\bigg|\int_{\Bbb S}\Lambda^qw\Lambda^q\bigg\{\frac{k_1}{3}w_x\Big((u^{\varepsilon}_x)^2+u_x^{\varepsilon}u_x^{\delta}
  +(u_x^{\delta})^2\Big)
 -\frac{k_1}{3}\mu\left(w_x\Big((u^{\varepsilon}_x)^2+u_x^{\varepsilon}u_x^{\delta} +(u_x^{\delta})^2\Big)\right)\bigg\}\,dx\bigg|\\
 &\qquad \qquad \le c\|w\|^2_{H^{q+1}}\left(\|u^{\varepsilon}\|^2_{H^{q+1}}+\|u^{\delta}\|^2_{H^{q+1}}\right),
 \end{align*}
 and
 \begin{align*}
 &\bigg|\int_{\Bbb S}\Lambda^qw\Lambda^q\Big\{\partial_x(\mu-\partial_x^2)^{-1}
\bigg(
k_1\Big[
2\mu(w)\mu(f)u^{\varepsilon}+2\mu^2(u^{\delta})w
+\mu(w)(u^{\varepsilon}_x)^2+\mu(u^{\delta})w_xf_x\Big]
\\&\qquad\qquad\qquad\qquad\qquad
+k_2\Big[2\mu(w)u^{\varepsilon}+2\mu(u^{\delta})w
+\frac1{2}w_xf_x
\Big]\bigg)\bigg\}\,dx\bigg|\\
&\le c \|w\|_{H^{q}}\bigg\|k_1\Big[
2\mu(w)\mu(f)u^{\varepsilon}+2\mu^2(u^{\delta})w
+\mu(w)(u^{\varepsilon}_x)^2+\mu(u^{\delta})w_xf_x\Big]\\&\qquad\qquad\qquad
+k_2\Big[2\mu(w)u^{\varepsilon}+2\mu(u^{\delta})w
+\frac1{2}w_xf_x
\Big]\bigg\|_{H^{q-1}}\\&\le c
\|w\|^2_{H^{q+1}}\left(
\|u^{\varepsilon}\|^2_{H^{q+1}}+\|u^{\delta}\|^2_{H^{q+1}}+\|u^{\varepsilon}\|_{H^{q+1}}
+\|u^{\delta}\|_{H^{q+1}}\right).
 \end{align*}
Therefore, we deduce that for any $\tilde{T}\in(0,1/(cM))$, there is a constant $c$ depending on $\tilde{T}$ such that
 \begin{align}\label{e3.7}
&\frac1{2}\dfrac{d}{dt}\left(
\|\Lambda^qw\|^2_{L^2}+\varepsilon\|\Lambda^qw_{x}\|^2_{L^2}
\right)\nonumber\\
&\le\,c\|w\|^2_{H^{q+1}}\left(\|u^{\varepsilon}\|^2_{H^{q+1}}
+\|u^{\delta}\|^2_{H^{q+1}}+\|u^{\varepsilon}\|_{H^{q+1}}+\|u^{\delta}\|_{H^{q+1}}\right)
\\&\quad+(\delta-\varepsilon)\|u_t^{\delta}\|_{H^{q+1}}\|w\|_{H^{q+1}}.\nonumber
\end{align}
Then, differentiating the equation in (\ref{m3.1}) with respect to $x$,
applying $\Lambda^qw_x\Lambda^q$ to both sides of the resulting equation
and integrating with respect to $x$ over ${\mathbb S}$, one finds that
\begin{align*}
\left|\int_{\Bbb S}\Lambda^qw_x\Lambda^qu^{\delta}_{
xxxt}\,dx\right|\le c\|u_t^{\delta}\|_{H^{q+3}}\|w\|_{H^{q+1}},
\end{align*}
and
 \begin{align*}
 &\bigg|\int_{\Bbb S}\Lambda^qw_x\Lambda^q\Big\{\partial^2_x(\mu-\partial_x^2)^{-1}
\bigg(
k_1\Big[
2\mu(w)\mu(f)u^{\varepsilon}+2\mu^2(u^{\delta})w
+\mu(w)(u^{\varepsilon}_x)^2+\mu(u^{\delta})w_xf_x\Big]
\\&\qquad\qquad\qquad\qquad\qquad
+k_2\Big[2\mu(w)u^{\varepsilon}+2\mu(u^{\delta})w
+\frac1{2}w_xf_x
\Big]\bigg)\bigg\}\,dx\bigg|\\
&\le c \|w_x\|_{H^{q}}\left\|k_1\Big[
2\mu(w)\mu(f)u^{\varepsilon}+2\mu^2(u^{\delta})w
+\mu(w)(u^{\varepsilon}_x)^2+\mu(u^{\delta})w_xf_x\Big]\right.\\&\qquad\qquad\qquad\left.
+k_2\Big[2\mu(w)u^{\varepsilon}+2\mu(u^{\delta})w
+\frac1{2}w_xf_x
\Big]\right\|_{H^{q}}\\&\le c
\|w\|^2_{H^{q+1}}\left(
\|u^{\varepsilon}\|^2_{H^{q+1}}+\|u^{\delta}\|^2_{H^{q+1}}+\|u^{\varepsilon}\|_{H^{q+1}}
+\|u^{\delta}\|_{H^{q+1}}\right).
 \end{align*}
For the rest terms of the right hand side, we just estimate one term
\begin{align*}
&\left|\int_{\Bbb S}\Lambda^qw_x\Lambda^q\left((u_x^{\delta})^2w_{x}\right)_x\,dx\right|\\
&=\Big|\int_{\Bbb S}\Lambda^{q+1}w\left(
[\Lambda^{q+1},(u_x^{\delta})^2]w_{x}+(u_x^{\delta})^2\Lambda^{q+1}w_{x}
\right)\,dx
-\int_{\Bbb S}\Lambda^{q}w\Lambda^{q}\left((u_x^{\delta})^2w_{x}\right)\,dx\Big|.
\end{align*}
Then applying the Cauchy-Schwartz inequality and Lemma \ref{l2.4}, we obtain
\begin{align*}
&\left|\int_{\Bbb S}\Lambda^qw_x\Lambda^q\left((u_x^{\delta})^2w_{x}\right)_x\,dx\right|
\\&\le c\|w\|_{H^{q+1}}\left(\|2u_x^{\delta}u_{xx}^{\delta}\|_{L^{\infty}}\|\Lambda^qw_{x}\|_{L^2}
+\|\Lambda^{q+1}(u_x^{\delta})^2\|_{L^2}\|w_{x}\|_{L^{\infty}}\right)\\&\qquad
+\|2u_x^{\delta}u_{xx}^{\delta}\|_{L^{\infty}}\|w\|^2_{H^{q+1}}
+c\|w\|^2_{H^{q+1}}\|u^{\delta}\|^2_{H^{q+1}}
\\&\le c\|w\|^2_{H^{q+1}}\|u^{\delta}\|^2_{H^{q+2}}.
\end{align*}
Therefore, one finds that for any $\tilde{T}\in(0,1/(cM))$,
there is a constant $c$ depending on $\tilde{T}$ such that
 \begin{align}\label{e3.8}
 &\frac1{2}\dfrac{d}{dt}\left(
 \|\Lambda^qw_{x}\|^2_{L^2}+\varepsilon\|w_{xx}\|^2_{H^q}
 \right)\nonumber\\
&\le\,c\|w\|^2_{H^{q+1}}\left(\|u^{\varepsilon}\|^2_{H^{q+2}}
+\|u^{\delta}\|^2_{H^{q+2}}+\|u^{\varepsilon}\|_{H^{q+1}}+\|u^{\delta}\|_{H^{q+1}}\right)
\\&\quad+(\delta-\varepsilon)\|u_t^{\delta}\|_{H^{q+3}}\|w\|_{H^{q+1}},\nonumber
 \end{align}
for any $t\in[0,\tilde{T})$. Altogether, we obtain
\begin{align*}
&\frac1{2}\dfrac{d}{dt}\Big(
\|\Lambda^qw\|^2_{L^2}
+\|\Lambda^qw_{x}\|^2_{L^2}+\varepsilon\|w_{x}\|^2_{H^q}+\varepsilon\|w_{xx}\|^2_{H^q}
\Big)\\
&\le c\|w\|^2_{H^{q+1}}\left(\|u^{\varepsilon}\|^2_{H^{q+2}}
+\|u^{\delta}\|^2_{H^{q+2}}+\|u^{\varepsilon}\|_{H^{q+1}}+\|u^{\delta}\|_{H^{q+1}}\right)\\
&\quad+(\delta-\varepsilon)\|w\|_{H^{q+1}}\left(\|u_t^{\delta}\|_{H^{q+1}}+\|u_t^{\delta}\|_{H^{q+3}}\right).
\end{align*}
With Lemma \ref{t3.3} and Lemma \ref{t3.4} in hand and in view of the fact that
$1/2<q<\min\{1,s-2\}$, integrating the above inequality
with respect to $t$ leads to the estimate
\begin{align*}
\|w\|^2_{H^{q+1}}\le 2\|w_0\|^2_{H^{q+1}}+2\varepsilon\|w_0\|^2_{H^{q+2}}+2c
\int_0^t(\delta^{\gamma}\|w\|_{H^{q+1}}+\|w\|^2_{H^{q+1}})\,d\tau,
\end{align*}
for any $t\in[0,\tilde{T})$, where $\gamma=1$ if $s\ge q+3$, and $\gamma=(4+3s-3q)/16$ if $s<q+3$.
It follows from Gronwall's inequality and \eqref{i3.3} in Lemma \ref{t3.3} that
\begin{align}\label{e3.9}
\|w\|_{H^{q+1}}&\le(2\|w_0\|^2_{H^{q+1}}+2\varepsilon\|w_0\|^2_{H^{q+2}})^{1/2}e^{ct}
+\delta^{\gamma}(e^{ct}-1)\nonumber\\
&\le c\delta^{(s-q-1)/16}e^{ct}+\delta^{\gamma}(e^{ct}-1)
\end{align}
for some constant $c$ and $t\in[0,\tilde{T})$.

Next, multiplying
$\Lambda^{s-1}w\Lambda^{s-1}$ to both sides of the equation in (\ref{m3.1}) and then
integrating with respect to $x$, following the way of the estimate \eqref{e3.7}, we obtain
\begin{align*}
&\frac1{2}\dfrac{d}{dt}\left(
\|\Lambda^{s-1}w\|^2_{L^2}+\varepsilon\|\Lambda^{s-1}w_{x}\|^2_{L^2}
\right)\\
&\le\,c\|w\|^2_{H^{s}}\left(\|u^{\varepsilon}\|^2_{H^{s}}
+\|u^{\delta}\|^2_{H^{s}}+\|u^{\varepsilon}\|_{H^{s}}+\|u^{\delta}\|_{H^{s}}\right)+(\delta-\varepsilon)\|u_t^{\delta}\|_{H^{s}}\|w\|_{H^{s}}.
\end{align*}
In the following, differentiating the equation in (\ref{m3.1}) with respect to $x$,
applying $\Lambda^{s-1}w_x\Lambda^{s-1}$ to both sides of the resulting equation
and integrating with respect to $x$ over ${\mathbb S}$, we note that
\begin{align*}
&\left|\int_{\Bbb S}\Lambda^{s-1}w_x\Lambda^{s-1}\left((u_x^{\delta})^2w_{x}\right)_x\,dx\right|
\\&=\Big|\int_{\Bbb S}\Lambda^{s}w\left(
[\Lambda^{s},(u_x^{\delta})^2]w_{x}+(u_x^{\delta})^2\Lambda^{s}w_{x}
\right)\,dx
-\int_{\Bbb S}\Lambda^{s-1}w\Lambda^{s-1}\left((u_x^{\delta})^2w_{x}\right)\,dx\Big|
\\&\le c\left(\|w\|^2_{H^{s}}\|u^{\delta}\|^2_{H^{s}}
+\|w\|_{H^{s}}\|w\|_{H^{q+1}}\|u^{\delta}\|_{H^{s}}\|u^{\delta}\|_{H^{s+1}}\right).
\end{align*}
Therefore, following the way of the estimate in \eqref{e3.8}, we deduce that
\begin{align*}
 &\frac1{2}\dfrac{d}{dt}\left(
 \|\Lambda^sw\|^2_{L^2}+\varepsilon\|w\|^2_{H^{s+1}}
 \right)\\
&\le\,c\|w\|^2_{H^{s}}\left(\|u^{\varepsilon}\|^2_{H^{s}}
+\|u^{\delta}\|^2_{H^{s}}+\|u^{\varepsilon}\|_{H^{s}}+\|u^{\delta}\|_{H^{s}}\right)
\\&\quad+c\|w\|_{H^{s}}\|w\|_{H^{q+1}}\|u^{\delta}\|_{H^{s}}\|u^{\delta}\|_{H^{s+1}}
+(\delta-\varepsilon)\|u_t^{\delta}\|_{H^{s+2}}\|w\|_{H^{s}},
 \end{align*}
for some constant $c$ and $t\in[0,\tilde{T})$. All in all, it
follows from the above estimates, and the inequality in \eqref{e3.9}
and Lemma \ref{t3.4} that there exists a constant $c$ depending on
the $\tilde{T}\in(0,1/(cM))$ such that
\begin{align*}
 &\frac1{2}\dfrac{d}{dt}\left(
 \|w\|^2_{H^s}+\varepsilon\|w\|^2_{H^{s+1}}
 \right)\\
&\le\,c\|w\|^2_{H^{s}}\left(\|u^{\varepsilon}\|^2_{H^{s}}
+\|u^{\delta}\|^2_{H^{s}}+\|u^{\varepsilon}\|_{H^{s}}+\|u^{\delta}\|_{H^{s}}\right)
\\&\quad+c\|w\|_{H^{s}}\|w\|_{H^{q+1}}\|u^{\delta}\|_{H^{s}}\|u^{\delta}\|_{H^{s+1}}
+(\delta-\varepsilon)\left(\|u_t^{\delta}\|_{H^{s}}+\|u_t^{\delta}\|_{H^{s+2}}\right)\|w\|_{H^{s}}
\\&\le\,c\left(\delta^{\alpha}\|w\|_{H^{s}}+\|w\|^2_{H^{s}}\right),
 \end{align*}
where $\alpha=\min\{7/16,(s-q-2)/16\}>0$.  Therefore, integrating
the above inequality with respect to $t$ leads to the estimate
 \begin{align*}
\|w\|^2_{H^{s}}&\le 2\|w_0\|^2_{H^{s}}+2\varepsilon\|w_0\|^2_{H^{s+1}}+2c
\int_0^t(\delta^{\alpha}\|w\|_{H^{s}}+\|w\|^2_{H^{s}})\,d\tau,
\end{align*}
By Gronwall's inequality and \eqref{i3.2} in Theorem \ref{t3.3}, we
have
 \begin{align*}
\|w\|_{H^{s}}&\le(2\|w_0\|^2_{H^{s}}+2\varepsilon\|w_0\|^2_{H^{s+1}})^{1/2}e^{ct}
+\delta^{\alpha}(e^{ct}-1)\\
&\le c(\|w_0\|_{H^{s}}+\delta^{7/16})e^{ct}+\delta^{\alpha}(e^{ct}-1).
\end{align*}
Making use of \eqref{i3.4} in Lemma \ref{t3.3} and the above
equality, we deduce that
 $\|w\|_{H^s}\rightarrow0$ as $\varepsilon,\delta\rightarrow0$.

At last, we prove convergence of $\{u_t^{\varepsilon}\}$.
Multiplying both sides of the equation in (\ref{m3.1}) by
$\Lambda^{s-1}w_t\Lambda^{s-1}$, and integrating the resulting
equation with respect to $x$, one obtains the equation using the integration by parts
 \begin{align*}
 \|w_t\|^2_{H^{s-1}}+\varepsilon\|w_t\|^2_{H^s}&\le\,c\|w_t\|_{H^{s-1}}\|w\|_{H^{s}}
 \left(\|u^{\varepsilon}\|^2_{H^{s}}
+\|u^{\delta}\|^2_{H^{s}}+\|u^{\varepsilon}\|_{H^{s}}+\|u^{\delta}\|_{H^{s}}\right)
\\&\quad+(\delta-\varepsilon)\|u_t^{\delta}\|_{H^{s+1}}\|w_t\|_{H^{s-1}}\\
&\le\delta^{3/4}\|w_t\|^2_{H^{s-1}}+c\|w_t\|_{H^{s-1}}\|w\|_{H^{s}},
 \end{align*}
for some constant $c$ and $t\in[0,\tilde{T})$. Hence,
\begin{align*}
\|w_t\|_{H^{s-1}}\le c(\delta^{3/4}+\|w\|_{H^{s}}).
\end{align*}
Then it follows that $w_t\rightarrow0$ as $\varepsilon,\delta\rightarrow0$ in $H^{s-1}$ norm. This implies that both
 $\{u^{\varepsilon}\}$ and $\{u^{\varepsilon}_t\}$ are Cauchy sequences in the space $C([0,\tilde{T});H^s)$ and $C([0,\tilde{T});H^{s-1})$,
 respectively. Let $u(t,x)$ be the limit of the sequence $\{u^{\varepsilon}\}$. Taking the limit on both sides of the equation
  in \eqref{e3.1} as $\varepsilon\rightarrow0$, one shows that $u$ is the solution of the problem \eqref{cauchy}.
\end{proof}
The verification for the uniqueness of the solution $u$ follows the
technique to obtain the norm $\|w\|_{H^q}$ in Theorem \ref{t3.5}.
\begin{theorem}\label{t3.6}
Suppose that $u_0\in H^s$, $s>5/2$. Then there exists a $T>0$, such
that the problem (\ref{cauchy}) has a unique solution $u(t,x)$ in
the space $C([0,T);H^s)\bigcap C^1([0,T);H^{s-1})$.
\end{theorem}
\begin{proof}
Suppose that $u$ and $v$ are two solutions of the problem
(\ref{cauchy}) corresponding to the same initial data $u_0$ such that
$u,v\in L^2([0,T);H^s)$. Then $w=u-v$ satisfies the following
Cauchy problem
\begin{equation}\label{d3.1}
\left\{
 \begin{aligned}
w_t
=-&2k_1\mu(w)uu_x
-\big(2k_1\mu(v)+k_2\big)\big(wu_x+vw_x\big)\\
+&\frac{k_1}{3}w_x\big(u_x^2+u_xv_x+v_x^2\big)
-\frac{k_1}{3}\mu\big(w_x(u_x^2+u_xv_x+v_x^2)\big)\\
-&\partial_x(\mu-\partial_x^2)^{-1} \bigg\{ k_1\Big[
2\mu(w)\mu(f)u+2\mu^2(v)w +\mu(w)u_x^2+\mu(v)w_xf_x\Big]
\\&\qquad\qquad\qquad\qquad
+k_2\Big[2\mu(w)u+2\mu(v)w
+\frac1{2}w_xf_x
\Big]\bigg\}, \\
w(0,x)&=0, \quad x \in \mathbb{R}.
\end{aligned} \right.
\end{equation}
where $f=u+v$. For any $1/2<q<\min\{1,s-2\}$. Applying the operator $\Lambda^q$ to both sides of the above
equation and then multiplying the resulting expression by
$\Lambda^qw$ to integrate with respect to $x$, one obtains the
following equation
\begin{align*}
\frac1{2}\frac{d}{dt}\|\Lambda^qw\|^2_{L^2}
&=-\int_{\mathbb{S}}\Lambda^qw\Lambda^q\bigg\{-2k_1\mu(w)uu_x
-\big(2k_1\mu(v)+k_2\big)\big(wu_x+vw_x\big)\\&
 \quad+\frac{k_1}{3}w_x\big(u_x^2+u_xv_x+v_x^2\big)
 -\frac{k_1}{3}\mu\big(w_x(u_x^2+u_xv_x+v_x^2)\big)\\&\quad
-\partial_x(\mu-\partial_x^2)^{-1} \bigg[ k_1\big(
2\mu(w)\mu(f)u+2\mu^2(v)w +\mu(w)u_x^2+\mu(v)w_xf_x\big)
\\&\qquad\qquad\qquad\qquad
+k_2\big(2\mu(w)u+2\mu(v)w +\frac1{2}w_xf_x \big)\bigg]\bigg\}\;dx.
\end{align*}
Similar to the proof of the estimate (\ref{e3.7}), there is a
constant $c$ such that
\begin{align*}
\frac{d}{dt}\|\Lambda^qw\|^2_{L^2}\le
c\|w\|^2_{H^{q+1}}\left(\|u\|^2_{H^{q+1}}+\|v\|^2_{H^{q+1}}+\|u\|_{H^{q+1}}+\|v\|_{H^{q+1}}\right).
\end{align*}
Next, differentiating the equation in \eqref{d3.1}, applying $\Lambda^qw_x\Lambda^q$ to both sides of the resulting equation and integrating over $\Bbb S$ with regard to $x$, one finds that
\begin{align*}
\frac1{2}\frac{d}{dt}\|\Lambda^qw_{x}\|^2_{L^2}
&=-\int_{\mathbb{S}}\Lambda^qw_x\Lambda^q\bigg\{-2k_1\mu(w)uu_x
-\big(2k_1\mu(v)+k_2\big)\big(wu_x+vw_x\big)\\&
 \quad+\frac{k_1}{3}w_x\big(u_x^2+u_xv_x+v_x^2\big)
 -\frac{k_1}{3}\mu\left(w_x\big(u_x^2+u_xv_x+v_x^2\big)\right)\\&\quad
-\partial_x(\mu-\partial_x^2)^{-1} \Big[
k_1\big(2\mu(w)\mu(f)u+2\mu^2(v)w +\mu(w)(u_x^2+\mu(v)w_xf_x\big)
\\&\qquad\qquad\qquad\qquad +k_2\big(2\mu(w)u+2\mu(v)w
+\frac1{2}w_xf_x \big)\Big]\bigg\}_x\;dx.
\end{align*}
Similar to the proof of the inequality \eqref{e3.8}, we deduce that there exists a constant $c$ such that
\begin{align*}
\frac{d}{dt}\|\Lambda^qw_{x}\|^2_{L^2}\le
c\|w\|^2_{H^{q+1}}\left(\|u\|^2_{H^{q+2}}+\|v\|^2_{H^{q+2}}+\|u\|_{H^{q+1}}+\|v\|_{H^{q+1}}\right).
\end{align*}
Combining the above two expressions, we deduce that
\begin{align*}
&\frac{d}{dt}\big(\|\Lambda^qw\|^2_{L^2}+\|\Lambda^qw_{x}\|^2_{L^2}\big)\\&\quad\le
c\|w\|^2_{H^{q+1}}\left(\|u\|^2_{H^{q+2}}+\|v\|^2_{H^{q+2}}+\|u\|_{H^{q+1}}+\|v\|_{H^{q+1}}\right).
\end{align*}
In view of $1/2<q<\min\{1,s-2\}$, then Gronwall's inequality and the
boundedness of $\|u\|_{H^{q+2}}$ and $\|v\|_{H^{q+2}}$ lead to
\begin{align*}
\|w\|_{H^{q+1}}\le\|w_0\|_{H^{q+1}}e^{\tilde{c}t}=0,
\end{align*}
for some constant $\tilde{c}$ and any $t\in(0,T)$. Hence,
$\|w\|_{H^{q+1}}=0$ and thus $w=0$ i.e. $u=v$.

Next, multiplying the equation in (\ref{d3.1}) by
$\Lambda^{s-1}w_t\Lambda^{s-1}$, and integrating the resulting
equation with respect to $x$, one obtains the inequality using the
integration by parts
\begin{align*}
  \|w_t\|_{H^{s-1}}\le c\|w\|_{H^{s}}\left(\|u\|^2_{H^{s}}+\|v\|^2_{H^{s}}\right).
\end{align*}
In view of $w=0$ and the boundedness of $\|u\|_{H^{s}}$ and
$\|v\|_{H^{s}}$, it follows that $\|w_t\|_{H^{s-1}}=0$ and thus
$w_t=0$ i.e. $u_t=v_t$. Hence the uniqueness of the solution to the
problem (\ref{cauchy}) is proved. The continuous dependency of
solutions on initial data can be verified by using a similar
technique used for the KdV equation by Bona and Smith \cite{bon}.
This completes the proof of Theorem \ref{t3-1}.
\end{proof}

%%%%%%%%%%%%%%%%%%%%%%%%%%%%%%%%%%%%%%%%%%%%%%%%%%%%%%%%%%%%
\renewcommand{\theequation}{\thesection.\arabic{equation}}
\setcounter{equation}{0}
%%%%%%%%%%%%%%%%%%%%%%%%%%%%%%%%%%%%%%%%%%%%%%%%%%%%%%%%%%%%

\section{Blow-up scenarios and a global conservative property}
In this section, attention  is now turned to the blow-up phenomena. We firstly present the following precise blow-up scenario.
\begin{theorem}\label{t4.1}
Let $ u_0\in H^{s + 2}$, $ s > 1/2$ be given and assume that $T$ is the
maximal existence time of the corresponding solution $u(t,x)$ to the
initial value problem (\ref{cauchy}) with the initial data $u_0$. Assume that
$T^*_{u_0}>0$ is the maximum time of existence. Then
\begin{align}\label{blow-criterion-2}
T^{\ast}_{u_0} < \infty \, \, \Rightarrow \, \,
\int_{0}^{T^{\ast}_{u_0}} \|m(\tau)\|_{L^{\infty}}^2\,d\tau =
\infty.
\end{align}
\end{theorem}

\begin{remark}\label{rmk-blow-criterion-2}
The blow-up criterion in \eqref{blow-criterion-2} implies that the
lifespan $T^{\ast}_{u_0}$ does not depend on the regularity index
$s$ of the initial data $u_0$. Indeed, let $u_0$ be in $H^s$ for
some $s>5/2$ and consider some $s' \in (5/2, \, s)$.
Denote by $u_{s}$ (resp., $u_{s'}$ ) the corresponding maximal $H^s$
(resp., $H^{s'}$ ) solution given by the above theorem. Denote by
$T^{\ast}_{s}$ (resp., $T^{\ast}_{s'}$) the lifespan of $u_{s}$
(resp., $u_{s'}$). Since $H^s \hookrightarrow H^{s'}$, uniqueness
ensures that $T^{\ast}_{s} \leq T^{\ast}_{s'}$ and that $u_{s}\equiv
u_{s'}$ on $[0, T^{\ast}_{s} )$. Now, if $T^{\ast}_{s} <
T^{\ast}_{s'}$, then we must have $u_{s'}$ in $C([0, T^{\ast}_{s}];
H^{s'})$. Hence, $u_{s'} \in L^2([0, T^{\ast}_{s}]; L^{\infty})$,
which contradicts the above blow-up criterion
\eqref{blow-criterion-2}. Therefore, $T^{\ast}_{s} = T^{\ast}_{s'}$.
\end{remark}

\begin{proof}
We shall prove
Theorem \ref{t4.1} by an inductive argument with
respect to the index $s$.  This will be carried out by three steps.

\vskip 0.2cm

\no {\bf Step 1.} For $s \in (1/2, 1)$, applying Lemma
\ref{lemtrans} to the equation in \eqref{cauchy}, i.e.,
\begin{eqnarray}\label{emmch}
m_t+\big(k_1(2\mu_0u-u_x^2)+k_2u\big)m_x=-2k_1u_xm^2-2k_2mu_x,
\end{eqnarray}
we arrive at
\begin{equation}\label{criterion-1-0}
\begin{split}
\|m(t)\|_{H^{s}}
 \leq & \
\|m_0\|_{H^{s}}+C\int_{0}^{t}\|\partial_x\big(k_1(2\mu_0u-u_x^2)+k_2u\big)(\tau)\|_{L^{\infty}}
\|m(\tau)\|_{H^{s}}\,d\tau\\
&+ 2C\int_{0}^{t}\|(u_xm^2+mu_x)(\tau)\|_{H^{s}}\,d\tau
\end{split}
\end{equation}
 for all $0 <t<T^{\ast}_{u_0}$. Owing to the Moser-type estimate in Lemma \ref{l2.3}, one has
\begin{equation}\label{criterion-1-1}
\|u_xm^2\|_{H^{s}} \leq
C(\|u_x\|_{H^{s}}\|m\|_{L^{\infty}}^2+\|u_x\|_{L^{\infty}}\|m\|_{L^{\infty}}\|m\|_{H^{s}}).
\end{equation}
An application of the Young inequality then implies
\begin{equation*}
\|u_x\|_{L^{\infty}}\leq \|g_x\|_{L^1}\cdot \|m\|_{L^{\infty}}\leq
\|m\|_{L^{\infty}}.
\end{equation*}
This esitmate together with the fact $\|u_x\|_{H^{s}} \leq C\|m\|_{H^{s}}$
and \eqref{criterion-1-1}, gives rise to
\begin{equation}\label{criterion-1-2}
\begin{split}
\|u_xm^2\|_{H^{s}}\leq C\|m\|_{H^{s}}\|m\|_{L^{\infty}}^2
\end{split}
\end{equation}
and
\begin{equation}\label{criterion-1-3}
\begin{split}
\|\partial_x\big(k_1(2\mu_0u-u_x^2)+k_2u\big)\|_{L^{\infty}} \leq &\; C\left(1+\|m\|_{L^{\infty}}\right)\|u_x\|_{L^{\infty}}\\
\leq &\;C\|m\|_{L^{\infty}}\left(1+\|m\|_{L^{\infty}}\right).
\end{split}
\end{equation}
Plugging \eqref{criterion-1-3} and \eqref{criterion-1-2} into
\eqref{criterion-1-0} leads to
\begin{equation}\label{criterion-1-4}
\begin{split}
\|m(t)\|_{H^{s}}
 \leq \|m_0\|_{H^{s}}+C\int_{0}^{t}
\|m(\tau)\|_{H^s} \|m(\tau)\|_{L^{\infty}}(1+ \|m(\tau)\|_{L^{\infty}})\,d\tau,
\end{split}
\end{equation}
which, by Gronwall's inequality, yields
\begin{equation}
\begin{split}\label{criterion-1-5}
\|m(t)\|_{H^{s}} \leq \|m_0\|_{H^{s}}
e^{C\int_{0}^{t}\|m(\tau)\|_{L^{\infty}}(1+ \|m(\tau)\|_{L^{\infty}})\,d\tau}.
\end{split}
\end{equation}
Therefore, if the maximal existence time $T^{\ast}_{u_0} < \infty $
satisfies
$$ \int_{0}^{T^{\ast}_{u_0}}
\|m(\tau)\|_{L^{\infty}}^2\,d\tau < \infty,$$
then
$$ \int_{0}^{T^{\ast}_{u_0}}
\left(\|m(\tau)\|_{L^{\infty}}+\|m(\tau)\|_{L^{\infty}}^2\right)\,d\tau \le \int_{0}^{T^{\ast}_{u_0}}
\left(\dfrac{1}{2}+\dfrac{3}{2}\|m(\tau)\|_{L^{\infty}}^2\right)\,d\tau <\infty.$$
Thus, the inequality
\eqref{criterion-1-5} implies that
\begin{equation}\label{criterion-1-6}
\begin{split}
\limsup_{t\rightarrow T^{\ast}_{u_0}}\|m(t)\|_{H^{s}} < \infty,
\end{split}
\end{equation}
which contradicts the assumption on the maximal existence time
$T^{\star}_{\mathbf{u}_0} < \infty .$ This completes the proof of
Theorem \ref{t4.1} for $s \in (1/2, 1)$.

\vskip 0.2cm

\no {\bf Step 2.} For $s \in [1, 2)$, by differentiating
the equation \eqref{emmch} with respect to $x$, we have
\begin{align}\label{equation-0-3}
\partial_t& (m_x)+\big(k_1(2\mu_0u-u_x^2)+k_2u\big)\partial_x(m_x)\nonumber\\&=
-3k_1u_x (m^2)_x-2k_1u_{xx}m^2-2k_2u_{xx}m-3k_2m_xu_x.
\end{align}
Applying Lemma \ref{lemtrans} to \eqref{equation-0-3} yields
\begin{equation}\label{criterion-2-1}
\begin{split}
\|m_x(t,\cdot)\|_{H^{s-1}} \leq &\;
\|m_{0,x}\|_{H^{s-1}}+C\int_{0}^{t}\|\partial_x\big(k_1(2\mu_0u-u_x^2)+k_2u\big)\|_{L^{\infty}}\|\pa_x
m\|_{H^{s-1}}d\tau\\
&+C \int_{0}^{t}\|3k_1u_x (m^2)_x+2k_1u_{xx}m^2+2k_2u_{xx}m+3k_2m_xu_x\|_{H^{s-1}}\,d\tau.
\end{split}
\end{equation}
By the Moser-type estimates in Lemma \ref{l2.3}, one has
\begin{equation}\label{criterion-2-2}
\|u_x (m^2)_x\|_{H^{s-1}} \leq C\|m\|_{L^{\infty}}^2\|m\|_{H^{s}}.
\end{equation}
and
\begin{equation}\label{criterion-2-4}
\|u_{xx}m^2\|_{H^{s-1}} \leq C\|m\|_{L^{\infty}}^2\| m\|_{H^{s-1}}.
\end{equation}
Using \eqref{criterion-1-3}, \eqref{criterion-2-2}, and \eqref{criterion-2-4} in
\eqref{criterion-2-1}, and combining with \eqref{criterion-1-4}, we
conclude that, for $1 \leq s <2$, \eqref{criterion-1-4} holds.

Repeating the same argument as in Step 1, we see that Theorem
\ref{t4.1} holds for $1 \leq s <2$.

\vskip 0.2cm

\no {\bf Step 3.} Suppose $2 \leq k \in \mathbb{N}$.  By induction, we
assume that \eqref{blow-criterion-2} holds when  $k-1 \leq s <k$,
and prove that it holds for $k \leq s <k+1$. To this end, we
differentiate \eqref{emmch} $k$ times with respect to $x$, producing
\begin{equation*}
\begin{aligned}
\pa_t\pa_x^{k}m&+\big(k_1(2\mu_0u-u_x^2)+k_2u\big)\pa_x(\pa_x^{k}m)\\
=&-\sum_{\ell=0}^{k-1}C_{k}^{\ell}
\partial_{x}^{k-\ell}\big(k_1(2\mu_0u-u_x^2)+k_2u\big) \partial_x^{\ell+1}m\\&
-2k_1\partial_x^{k}(u_xm^2)-2k_2\partial_x^{k}(mu_x).
\end{aligned}
\end{equation*}
Applying Lemma \ref{lemtrans} to the above equation again yields
that
\begin{equation}\label{criterion-3-1}
\begin{split}
\|\pa_x^{k}m(t)\|_{H^{s-k}}& \leq \|\pa_x^{k}m_0\|_{H^{s-k}}+C
\int_{0}^{t}\|\pa_x^{k}m(\tau)\|_{H^{s-k}}\|(2k_1mu_x+k_2u_x)(\tau)\|_{L^{\infty}}\,d\tau\\
&\ +C\int_{0}^{t}\left\|\left(\sum_{\ell=0}^{k-1}C_{k}^{\ell}
\partial_{x}^{k-\ell}\big(k_1(2\mu_0u-u_x^2)+k_2u\big) \partial_x^{\ell+1}m
 \right)(\tau)\right\|_{H^{s-k}}\,d\tau\\&\,+C\int_{0}^{t}\left\|\left(2k_1\partial_x^{k}(u_xm^2)
 +2k_2\partial_x^{k}(u_xm)\right)(\tau)\right\|_{H^{s-k}}\,d\tau.
\end{split}
\end{equation}
Using the Moser-type estimate in Lemma \ref{l2.3} and the
Sobolev embedding inequality, we have
\begin{equation*}
\begin{split}
&\left\|\sum_{\ell=0}^{k-1}C_{k}^{\ell}
\partial_{x}^{k-\ell} \big(k_1(2\mu_0u-u_x^2)+k_2u\big) \partial_x^{\ell+1}m\right \|_{H^{s-k}}\\
&\quad\leq C\sum_{\ell=0}^{k-1}C_{k}^{\ell}\big[\|k_1(2\mu_0u-u_x^2)+k_2u\|_{H^{s-\ell+1}}
\|\partial_x^{\ell}m\|_{L^{\infty}}\\
&\;\qquad\qquad \qquad +\|\partial_x^{k-\ell}\big(k_1(2\mu_0u-u_x^2)+k_2u\big)\|_{L^{\infty}}\|m\|_{H^{s-k+\ell+1}}\big]\\
&\quad\leq  C \|m\|_{H^{s}}\|m\|_{H^{k-\frac{1}{2}+\varepsilon_0}}\left(1+\|m\|_{H^{k-\frac{1}{2}+\varepsilon_0}}\right),
\end{split}
\end{equation*}
where the genius constant $\varepsilon_0 \in (0, \,\frac{1}{4})$ so
that $ H^{\frac{1}{2}+\varepsilon_0}  \hookrightarrow L^{\infty}$ holds.
Substituting these estimates into
\eqref{criterion-3-1}, we obtain
\begin{equation}\label{criterion-3-4}
\begin{split}
\|m(t)\|_{H^{s}} \leq \|m_{0}\|_{H^{s}}&
+C\int_{0}^{t}\left(1+\|m(\tau)\|_{H^{k-\frac{1}{2}+\varepsilon_0}}\right)
\|m(\tau)\|_{H^{k-\frac{1}{2}+\varepsilon_0}}\|m(\tau)\|_{H^{s}}\,d\tau,
\end{split}
\end{equation}
where we used the Sobolev embedding theorem
$H^{k-\frac{1}{2}+\varepsilon_0}\hookrightarrow L^{\infty}$ with $k
\geq 2$. Applying Gronwall's inequality then gives
\begin{equation}\label{criterion-3-5}
\begin{split}
\|m(t)\|_{H^{s}}
 \leq \|m_0\|_{H^{s}}
 \exp\{C\int_{0}^{t}\|m(\tau)\|_{H^{k-\frac{1}{2}+\varepsilon_0}}
 \left(1+\|m(\tau)\|_{H^{k-\frac{1}{2}+\varepsilon_0}}\right)\,d\tau\}.
\end{split}
\end{equation}
In consequence,  if the maximal existence time
$T^{\star}_{\mathbf{u}_0} < \infty $ satisfies
\begin{equation*}
\int_{0}^{T^{\star}_{\mathbf{u}_0}}
\|m(\tau)\|_{L^{\infty}}^2\,d\tau < \infty,
\end{equation*} thanks to the uniqueness of
solution in Theorem \ref{t3-1}, we then find that
$\|m(t)\|_{H^{k-\frac{1}{2}+\varepsilon_0}}$ is uniformly bounded in
$t \in (0, T^{\ast}_{\mathbf{u}_0})$ by the induction assumption,
which along with \eqref{criterion-3-5} implies
\begin{equation*}
\begin{split}
\limsup_{t\rightarrow T^{\ast}_{\mathbf{u}_0}}\|m(t)\|_{H^{s}} <
\infty,
\end{split}
\end{equation*}
which leads to a contradiction.  The
result of Theorem \ref{t4.1} then follows from Step 1 to Step 3.
\end{proof}
\begin{remark}
In fact, we can improve the result of Theorem \ref{t4.1} as follows:
\begin{equation}\label{im4.1}
T^{\ast}_{u_0} < \infty \, \, \Rightarrow \, \,
\int_{0}^{T^{\ast}_{u_0}} \|\left(k_1mu_x+k_2u_x\right)(\tau)\|_{L^{\infty}}\,d\tau =
\infty.
\end{equation}
Applying the maximum principle to the transport equation \eqref{emmch}, we immediately get
\begin{align*}
\|m(t)\|_{L^{\infty}} &\le\|m_0\|_{L^{\infty}} + C\int_0^{t}\|\left(k_1mu_x+k_2u_x\right)(\tau)\|_{L^{\infty}}\|m(\tau)\|_{L^{\infty}}\,d\tau.
\end{align*}
 Then, Gronwall's inequlity applied to the above inequality yields
\begin{equation*}
\|m(t)\|_{L^{\infty}} \leq \|m_0\|_{L^{\infty}} \exp \left( C\int_0^{t}\|\left(k_1mu_x+k_2u_x\right)(\tau)\|_{L^{\infty}}\,d\tau\right),
\end{equation*}
which along with the result of Theorem \ref{t4.1} gives rise to \eqref{im4.1}.
\end{remark}

In order to demonstrate a conservative property, let us consider the
trajectory equation
\begin{equation}\label{flow-1}
\left\{
 \begin{array}{ll}
\begin{split}
&\dfrac{dq}{dt}=\Big(k_1(2\mu_0u-u_x^2)+k_2u\Big)(t,q(t,x))\\
&q(0,x)=x.
 \end{split}
\end{array} \right.
\end{equation}
\begin{lemma}\label{lem-flow-1}
Let $u_{0}\in H^{s}, \; s > 5/2$, and let $T>0$ be the
maximal existence time of the corresponding strong solution $u$ to
\eqref{cauchy}. Then \eqref{flow-1} has a unique solution $q \in
C^{1}([0,T)\times {\Bbb S})$ such that the map $q(t,\cdot)$ is an
increasing diffeomorphism over ${\Bbb S}$ with
\begin{equation}\label{diff-flow-1}
q_{x}(t,x)=\exp \left ( \int_{0}^{t}(2k_1mu_{x}+k_2u_x)(s,q(s,x))ds \right ) >
0,\,\,\hbox{for all} \quad (t,x)\in [0,T)\times {\Bbb S}.
\end{equation}
Furthermore,
\begin{equation}\label{diff-flow-2}
m(t,q(t,x))=m_{0}(x)\exp\left(-2\int_0^t(k_1mu_x+k_2u_x)(s,q(s,x))\,ds\right)
\end{equation}
for all  $(t,x)\in[0,T)\times{\Bbb S}$.
\end{lemma}
\begin{proof}
Since $u \in C^1 \left([0, T), H^{s-1}({\Bbb S})\right)$ and
$H^{s}\hookrightarrow C^1,$ both $u(t, x)$ and $u_x(t, x)$ are
bounded, Lipschitz in the space variable $x$, and of class $C^1$ in
time. Therefore, by well-known classical results in the theory of
ordinary differential equations, the initial value problem
\eqref{flow-1} has a unique solution $q(t, x) \in C^1 \left([0, T)
\times {\Bbb S}\right).$

Differentiating \eqref{flow-1} with respect to $x$ yields
\begin{equation*}
\begin{cases}\frac{d}{dt}q_{x}=(2k_1mu_x+k_2u_x)(t,q)q_x,
\\[.2cm] q_x(0,x) = 1,
\end{cases}
  \qquad x\in {\Bbb S},\quad t\in[0,T).
\end{equation*}
The solution to the above initial-value problem is given by
\begin{equation*}
q_{x}(t,x)=\exp \left ( \int_{0}^{t}(2k_1mu_x+k_2u_x)(s,q(s,x))ds \right
),\, \quad (t,x)\in [0,T)\times {\Bbb S}.
\end{equation*}
For every $T' < T,$ it follows from the Sobolev embedding theorem
that
\begin{equation*}
\sup_{(s,x)\in [0,T')\times {\Bbb S}}|(2k_1mu_x+k_2u_x)(s, x)|< \infty.
\end{equation*} We infer from the expression of $q_x$
that there exists a constant $K > 0$ such that $q_x(t, x) \geq
e^{-Kt}, \ (t,x)\in [0,T)\times {\Bbb S},$ which implies that the
map $q(t,\cdot)$ is an increasing diffeomorphism of ${\Bbb S}$
before blow-up with
\begin{equation*}
q_{x}(t,x)=\exp \left ( \int_{0}^{t}(2k_1mu_x+k_2u_x)(s,q(s,x))ds \right ) >
0,\, \quad \hbox{for all} \quad  (t,x)\in [0,T)\times {\Bbb S}.
\end{equation*}
On the other hand, from \eqref{emmch} we
have
\begin{align*}
\dfrac{d}{dt}m(t,q(t,x))=&\ m_t(t,q)+\Big(k_1(2\mu_0u-u_x^2)+k_2u\Big)(t,q(t,x))m_x(t,q(t,x))\\
=&\;-2\Big(k_1mu_x+k_2u_x\Big)(t,q(t,x))m(t,q(t,x)).
\end{align*}
Therefore, solving the equation with regard to $m(t,q(t,x))$ leads to
\begin{align*}
m(t,q(t,x))=m_{0}(x)\exp\left(-2\int_0^t(k_1mu_x+k_2u_x)(s,q(s,x))\,ds\right).
\end{align*}
This completes the proof of Lemma \ref{lem-flow-1}.
\end{proof}
\begin{remark}\label{rmk-flow-1}
Lemma \ref{lem-flow-1} shows that, if $m_0 =(\mu-\partial_x^2) u_0$
does not change sign, then $m(t,x)$  will not change sign for any $
t \in [0, T)$.
\end{remark}

\begin{remark}\label{rmk-flow-2}
Since $q(t,\cdot)\colon {\Bbb S}\rightarrow{\Bbb S}$ is a
diffeomorphism of the line for every $t\in[0,T)$, the
$L^{\infty}$-norm of any function $v(t,\cdot)\in L^{\infty}$ is
preserved under the family of diffeomorphisms $q(t,\cdot)$, that is,
\begin{equation*}
\|v(t,\cdot)\|_{L^{\infty}}=\|v(t,q(t,\cdot))\|_{L^{\infty}},\quad
t\in[0,T).
\end{equation*}
\end{remark}

The following blow-up criterion implies that wave-breaking depends
only on the infimum of $k_1m\,u_x$ or $k_2u_x$.

\begin{theorem}\label{cor-blow-criterion-2}
Let $u_0\in H^s$, $ s > 5/2 $ be as in Theorem \ref{t4.1}.  Then the corresponding solution $u$ to \eqref{cauchy}
blows up in finite time $ T^{\ast}_{u_0}
> 0 $ if and only if
\begin{equation}\label{blow-1}
\underset{t\uparrow T^{\ast}_{u_0}}{\liminf}\left (\underset{x\in
{\Bbb S}}{\inf} (k_1mu_x(t,x)+k_2u_x(t,x)) \right )=-\infty.
\end{equation}
\end{theorem}
\begin{proof}
In view of the equation in \eqref{diff-flow-2},  if $\underset{x\in{\Bbb S}}{\inf} \{k_1mu_x+k_2u_x\}\ge-C_1$  for $0 \le t \le T^*_{u_0}$,
then it implies that
\begin{align*}
\|m(t)\|^2_{L^{\infty}}&=\|m(t,q(t,x))\|^2_{L^{\infty}}=\|m_0(x)\|^2_{L^{\infty}}
\exp\left(-4\int_0^t(k_1mu_x+k_2u_x)(s,q(s,x))\,ds\right)\\&\le e^{4C_1}\|m_0(x)\|^2_{L^{\infty}}.
\end{align*}
Thanks to Theorem \ref{t4.1}, it ensures that the solution $m(t,x)$ does not blow up in finite time.

On the other hand, if
$$
\underset{t\uparrow T^{\ast}_{u_0}}{\liminf}\left (\underset{x\in
{\Bbb S}}{\inf} (k_1mu_x(t,x)+k_2u_x(t,x)) \right )=-\infty,
$$
 by Theorem \ref{t3-1} for the existence of local strong
solutions and the Sobolev embedding theorem, we infer that the
solution will blow up in finite time. The proof of the theorem is then complete.
\end{proof}
\begin{cor}\label{thm-blow-criterion-2}
Let $u_0\in H^s$ be as in Theorem \ref{t4.1} with $ s>
5/2$. Then the corresponding solution $u$ to \eqref{cauchy}
blows up in finite time $ T^{\ast}_{u_0}
> 0 $ if and only if
\begin{equation}\label{blow-1}
\underset{t\uparrow T^{\ast}_{u_0}}{\liminf} \left ( \min \{ \underset{x\in
{\Bbb S}}{\inf} (k_1mu_x(t,x)), \underset{x\in
{\Bbb S}}{\inf} (k_2u_x(t,x))  \} \right ) = -\infty.
\end{equation}
\end{cor}
\begin{proof}
It is observed that for any $ t \in [0, T) $
$$\underset{x\in
{\Bbb S}}{\inf} \left ( k_1mu_x(t,x)+k_2u_x(t,x) \right ) \ge 2  \min \{ \underset{x\in
{\Bbb S}}{\inf} (k_1mu_x(t,x)), \underset{x\in
{\Bbb S}}{\inf} (k_2u_x(t,x)) \}.$$
If there exists a positive constant $C$ such that
$$\min \{ \underset{x\in
{\Bbb S}}{\inf} (k_1mu_x(t,x)), \underset{x\in
{\Bbb S}}{\inf} (k_2u_x(t,x)) \}  \ge-C,
$$
then
$$\underset{x\in
{\Bbb S}}{\inf} \left ( k_1mu_x(t,x)+k_2u_x(t,x) \right ) \ge-2C.$$
This thus implies  from  Theorem \ref{cor-blow-criterion-2} that the solution $u(t,x)$ does not blow up in finite time.

On the other hand, similar to the proof of Theorem \ref{cor-blow-criterion-2}, if
$$
\underset{t\uparrow T^{\ast}_{u_0}}{\liminf} \left ( \min \{ \underset{x\in
{\Bbb S}}{\inf} (k_1mu_x(t,x)), \underset{x\in
{\Bbb S}}{\inf} (k_2u_x(t,x))  \} \right ) = -\infty,
$$
then the solution will blow-up in finite time.
\end{proof}

%%%%%%%%%%%%%%%%%%%%%%%%%%%%%%%%%%%%%%%%%%%%%%%%%%%%%%%%%%%
\renewcommand{\theequation}{\thesection.\arabic{equation}}
\setcounter{equation}{0}
%%%%%%%%%%%%%%%%%%%%%%%%%%%%%%%%%%%%%%%%%%%%%%%%%%%%%%%%%%%%

\section {Wave-breaking mechanism}

In this section, we give some sufficient conditions for the
breaking of waves for the initial-value problem \eqref{cauchy}.
In the case of $ \mu_0 = \mu(u_0) = 0, $ we have the following blow-up result.

\begin{thm}\label{blow-up-1}
Let $k_1>0$, $k_2>0$ and $u_0\in H^s$, $s > 5/2$, $ u_0 \neq 0 $  with $ \mu_0 = 0$.  Let $ T > 0 $ be the
maximal time of existence of the corresponding solution $u(t,x)$
of \eqref{cauchy} with the initial value $ u_0. $ Then the solution $ u $
blows up in  finite time $ T^* < \infty. $
\end{thm}
\begin{proof} In view of  Theorem \ref{thm-blow-criterion-2} with the density argument, it suffices
to consider the case $s\geq3$.
Since $ \mu(u)$ is conserved, so we have $\mu(u)=0$. Thus the equation in \eqref{emmch} is reduced to
\begin{eqnarray}\label{blow-up-1-1}
u_{xxt}-k_1\left(u_x^2u_{xx}\right)_x+k_2(2u_xu_{xx}+uu_{xxx})=0.
\end{eqnarray}
Integrating it with respect to $x$ from $0$ to $x$, we arrive at
\begin{equation} \label{blow-up-1-2}
u_{xt}+(-k_1u_x^2+k_2u)u_{xx}+\frac{k_2}{2}u^2_x= -\frac{k_2}{2}\mu_1^2.
\end{equation}
Notice that $u\in C([0,T); H^s)$ is a periodic function of $x$.
Then there exists a $x_0\in {\Bbb S}$ such that
$u_{0,x}(x_0)<0$. Consider the flow
\eqref{flow-1} with $\mu_0=0$, i.e., the flow $q(t,x)$ governed by
$-k_1u_x^2+k_2u,$
\begin{eqnarray}\label{blow-up-3}
\frac{dq}{dt}=(-k_1u_x^2+k_2u)(t,q(t,x)),\;\; q|_{t=0}=x.
\end{eqnarray}
Along with the flow $ q(t, x)$, we deduce from (\ref{blow-up-1-2})
that
\begin{eqnarray*}
\frac{du_x (t, q(t, x))}{dt}=u_{xt}+(-k_1u_x^2+k_2u)u_{xx}=-\fr {k_2}{2}(u_x^2+\mu_1^2).
\end{eqnarray*}
Let $w(t)=k_2u_{x}(t, q(t,x_0))$. It follows from the above equation that
\begin{eqnarray*}
\frac {dw}{dt}=-\fr 1{2} (w^2+k_2^2 \mu_1^2).
\end{eqnarray*}
Solving it we get
\begin{eqnarray*}
w(t)=k_2u_{x}(t,q(t, x_0))=\mu_1k_2\dfrac{u_{0,x}(x_0)-\mu_1\tan(\frac1{2}\mu_1k_2t)}
{\mu_1+u_{0,x}(x_0)\tan(\frac1{2}\mu_1k_2t)}.
\end{eqnarray*}
It is thereby inferred that
$$
\inf_{x \in \mathbb{S}} (k_2 u_x)(t, x) \le k_2 u_x(t, q(t, x_0))
\longrightarrow - \infty,  \quad {\rm as} \; t \to T^* \le t^*,
$$
where $t^*=-2\arctan(\frac{\mu_1}{u_{0,x}(x_0)})/(\mu_1k_2) < \infty$.

Let $U(t)=k_1u_{xx}(t, q(t,x_0))$. It follows from \eqref{blow-up-1-1} and \eqref{blow-up-3} that
\begin{eqnarray*}
\frac {dU}{dt}=k_1u_{xxt}+(-k_1u_x^2+k_2u)k_1u_{xxx}=2u_x(U^2-k_2U).
\end{eqnarray*}
Solving it we get
\begin{eqnarray*}
U(t)=k_1u_{xx}(t,q(t, x_0))=\dfrac{k_2U(0)}{U(0)-(U(0)-k_2)\left(\cos(\frac1{2}\mu_1k_2t)
+\frac{u_{0,x}(x_0)}{\mu_1}\sin(\frac1{2}\mu_1k_2t)\right)^4}.
\end{eqnarray*}
Accordingly, we have
\begin{align*}
k_1(mu_{x})(t,q(t, x_0))=&\;-k_1(u_xu_{xx})(t,q(t, x_0))\\=&\;-\dfrac{\mu_1k_1k_2u_{0,xx}(x_0)\left(u_{0,x}(x_0)-\mu_1\tan(\frac1{2}\mu_1k_2t)\right)}
{\left(\mu_1+u_{0,x}(x_0)\tan(\frac1{2}\mu_1k_2t)\right) J_1},
\end{align*}
where
\begin{eqnarray*}
J_1=k_1u_{0,xx}(x_0)-(k_1u_{0,xx}(x_0)-k_2)\cos^4(\frac1{2}\mu_1k_2t)
\Big(1+\frac{u_{0,x}(x_0)}{\mu_1}\tan(\frac1{2}\mu_1k_2t)\Big)^4.
\end{eqnarray*}
It is thereby inferred that
$$
 k_1(mu_x)(t, q(t, x_0))\longrightarrow + \infty  \quad {\rm as} \quad  t \to T^* \le t^*.
$$
Moreover, a simple computation shows that for any constant $ a $
 \begin{equation}\label{blow-up-7}
 \begin{split}
&(k_1 u_x m +ak_2u_x)(t, q(t, x_0))\\
&\;=\frac{\tan\alpha}{J_2}\Big[k_1(1-a)\tan^4\alpha+2k_1(1-a)\tan^2\alpha+k_1\\
&\;\qquad +ak_2\mu_1(1+\tan^2\theta)^2\tan\theta (u_{0,x}m_0)^{-1}+ak_1(\tan^4\theta+2\tan^2\theta)\Big],
\end{split}
\end{equation}
where
\begin{eqnarray*}
\begin{aligned}
 &\; J_2=\frac{k_1}{k_2\mu_1}\Big[(\tan^4\alpha+2\tan^2\alpha)-(1+\tan^2\theta)^2\tan\theta (u_{0,x}m_0)^{-1}-(\tan^4\theta+2\tan^2\theta)\Big],\\
&\alpha=\fr{1}{2}k_2\mu_1t-\theta, \quad
\theta = \arctan \left (\frac{u_{0,x}(x_0)}{\mu_0} \right ).
\end{aligned}
\end{eqnarray*}
When  $a>1$, it then follows from \eqref{blow-up-7} that
$$(k_1 u_x m +ak_2u_x)(t, q(t, x_0)) \rightarrow-\infty, \;\; {\rm as}\;\; t\rightarrow T^* \le t^*,
$$which implies the desired result.
%Consequently, we have for $ a > 1 $
%$$
%\underset{t\uparrow t^*}{\liminf}(\inf_{x \in \mathbb{S}} (k_1 u_x m + a k_2 u_x)(t,x)) = - \infty. $$
\end{proof}

To deal with blow-up solution with $ \mu_0 \neq 0, $ the following two lemmas will be useful.
\begin{lem}\label{prop-blow-1}
Let $u_0\in H^s, s \ge 3$ and $T>0$ be the maximal time
of existence of the corresponding solution $u(t,x)$ to
\eqref{cauchy} with the initial data $u_0$.
Then  $M=m u_x $, $P=u_x$ and $\Gamma=(k_1 m+k_2)u_x$ satisfy
\begin{subequations}\label{blow-1-0}
 \begin{equation}\label{blow-1-0-1}
 \begin{split}
M_{t}+&\left(k_1(2\mu_0u-u_x^2)+k_2u\right)M_x\\
=&\;-2k_1M^2-\fr{5}{2}k_2u_x M +k_1\left[2\mu_0^2(u-\mu_0)-\mu_0(u_x^2+\mu_1^2)\right]m\\
&\;+k_2\left[2\mu_0(u-\mu_0)-\fr{1}{2}\mu_1^2\right]m,
 \end{split}
 \end{equation}
 \begin{equation}\label{blow-1-0-2}
 \begin{split}
\quad P_{t}+&\left(k_1(2\mu_0u-u_x^2)+k_2u\right)P_x\\
=&\;k_1\left[2\mu_0^2(u-\mu_0)-\mu_0(u_x^2+\mu_1^2)\right]+k_2\left[2\mu_0(u-\mu_0)-\fr{1}{2}(u_x^2+\mu_1^2)\right],
 \end{split}
 \end{equation}
 and
 \begin{eqnarray}\label{blow-1-0-3}
 \begin{aligned}
\qquad \quad\Gamma_{t}+&\left(k_1(2\mu_0u-u_x^2)+k_2u\right)\Gamma_x\\
=&\;-2\Gamma^2+(\fr 32 k_2-k_1\mu_0)u_x\Gamma \\
&\; +(k_1m+k_2)\Big[k_1\left(2\mu_0^2(u-\mu_0)-\mu_0\mu_1^2 \right)+k_2 \left(2\mu_0(u-\mu_0)-\fr{1}{2}\mu_1^2\right)\Big].
\end{aligned}
\end{eqnarray}
\end{subequations}
\end{lem}

\begin{proof} We will give a proof of the last estimate \eqref{blow-1-0-3} only, since the other cases can be obtained similarly.
In view of \eqref{emmch}, a direct computation
shows
 \begin{eqnarray*}
 \begin{split}
(\mu-\partial_x^2)&\left[u_{t}+ \big(k_1(2\mu_0u-u_x^2)+k_2u\big)u_x\right]\\
 =&\;m_t-k_1\mu(u_x^3)-\left[k_1(2\mu_0u-u_x^2)u_{xx}+2k_1u_x^2m+k_2(uu_{xx}+u_x^2)\right]_x\\
 =&\;k_1\left(-2\mu_0^2u_x-2\mu_0u_xu_{xx}+4u_xu_{xx}^2+2u_x^2u_{xxx}-\mu(u_x^3)\right)\\
 &\;-k_2(2\mu_0u_x+u_xu_{xx}),
 \end{split}
\end{eqnarray*}
which implies
 \begin{eqnarray}\label{blow-1-1}
 \begin{split}
&u_{t}+ \big(k_1(2\mu_0u-u_x^2)+k_2u\big)u_x\\
&\;=(\mu-\partial_x^2)^{-1}
\big[k_1\left(-2\mu_0^2u_x-2\mu_0u_xu_{xx}+4u_xu_{xx}^2+2u_x^2u_{xxx}-\mu(u_x^3)\right)\\
 &\qquad -k_2(2\mu_0u_x+u_xu_{xx})\big].
\end{split}
\end{eqnarray}
Differentiating it with respect to $x$ leads to
 \begin{equation}\label{blow-1-2}
 \begin{split}
u_{xt}+&k_1\left((2\mu_0u-u_x^2)u_{xx}+2u_x^2m\right)+k_2(uu_{xx}+u_x^2)\\
=&\;k_1\left[2\mu_0^2(u-\mu_0)+\mu_0(u_x^2-\mu_1^2)-2u_x^2u_{xx}\right]\\
&\;+k_2\left[2\mu_0(u-\mu_0)+\fr{1}{2}(u_x^2-\mu_1^2)\right].
 \end{split}
\end{equation}
Clearly,
\begin{equation}\label{blow-1-3}
u_x m_t+\left(k_1(2\mu_0u-u_x^2)+k_2u\right)m_xu_x=-2k_1u_x^2m^2-2k_2m u_x^2.
\end{equation}
Multiplying \eqref{blow-1-2} by $m$ and adding the resulting one
with \eqref{blow-1-3}, we arrive at
\begin{equation}\label{blow-1-6}
 \begin{split}
M_{t}+&\left(k_1(2\mu_0u-u_x^2)+k_2u\right)M_{x}\\
=&\;-4k_1M^2-3k_2u_xM+k_1\left[2\mu_0^2(u-\mu_0)+\mu_0^2(u_x^2-\mu_1^2)-2u_x^2u_{xx}\right]m\\
&\;+k_2\left[2\mu_0(u-\mu_0)+\fr{1}{2}(u_x^2-\mu_1^2)\right]m.
 \end{split}
\end{equation}
Combining \eqref{blow-1-2} and \eqref{blow-1-6}, we obtain
\begin{equation}\label{blow-1-7}
 \begin{split}
\Gamma_{t}+&\left(k_1(2\mu_0u-u_x^2)+k_2u\right)\Gamma_{x}\\
=&\;-2k_1^2M^2-\fr{5}{2}k_1k_2u_xM-\fr{1}{2}k_2^2u_x^2+k_1^2\left[2\mu_0^2(u-\mu_0)-\mu_0(u_x^2+\mu_1^2)\right]m\\
&\;+k_1k_2\left[2\mu_0(u-\mu_0)-\fr{1}{2}\mu_1^2\right]m+ k_1k_2\left[2\mu_0^2(u-\mu_0)-\mu_0(u_x^2+\mu_1^2)\right]\\
&\;+k_2^2\left[2\mu_0(u-\mu_0)-\fr{1}{2}\mu_1^2\right].
 \end{split}
\end{equation}
A simply computation then gives the desired result.
\end{proof}

\begin{lem}\label{lem-blow-2} Let $T>0$ be the maximal time
of existence of the solution $u(t,x)$ to the Cauchy problem
\eqref{cauchy} with initial data $u_0\in H^s, s \ge 3$. Assume
$k_1\geq0$, $k_2\geq0$ and $m_0(x)\geq 0$ for all $x\in {\Bbb S}$. Then there hold
 \begin{equation}\label{blow-2-0}
 \begin{split}
 &\;\;|u_x(t,x)|\leq u(t,x),\\
&\; \Gamma_{t}+\left(k_1(2\mu_0u-u_x^2)+k_2u\right)\Gamma_x\leq -2\Gamma^2+C_1 u_x\Gamma +C_2 (k_1m+k_2),
\end{split}
\end{equation}
where the constants $ C_1 $ and $ C_2 $ are given by
\begin{eqnarray*}
\begin{aligned}
C_1=&\;\fr32 k_2-k_1 \mu_0,\\
C_2=&\;k_1\mu_0\mu_1(\fr {\sqrt{3}}{3}\mu_0-\mu_1)+k_2\mu_1(\fr {\sqrt{3}}3\mu_0-\fr 12\mu_1).
\end{aligned}
\end{eqnarray*}
\end{lem}

\begin{proof} Since $m_0 \ge 0$, Lemma \ref{lem-flow-1}
implies that $m(t,x) \ge 0$ for any $t>0$, $x\in{\Bbb S}$. According to \eqref{g},
there holds
\begin{align*}
u(t,x)=g\ast m=\int_{{\Bbb S}}[\fr 12 (x-y-\fr
{1}{2})^2+\fr{23}{24}]m(t,y)dy>0,
\end{align*}
for any $t>0$, $x\in {\Bbb S}$. Note that $\mu_0=\mu(u)>0$, and
$ u_x=\int_{{\Bbb S}}(x-y-\fr {1}{2})m(t,y)dy. $ It follows that
\begin{align*}
u+u_x=&\int_{{\Bbb S}}[\fr 1{2}(x-y+\fr
{1}{2})^2+\fr{11}{24}]m(t,y)dy>0,\\
u-u_x=&\int_{{\Bbb S}}[\fr 1{2}(x-y-\fr
{3}{2})^2+\fr{11}{24}]m(t,y)dy>0.\\
\end{align*}
This implies
\begin{equation*}
 |u_x(t, x)| \leq u(t, x),
\end{equation*}
for all $(t, x) \in [0, T) \times {\Bbb S}$. By Lemma \ref{sobolev-3}, we have
\begin{equation}
 \|u-\mu_0\|_{L^{\infty}}\leq \fr{\sqrt{3}}{6}\mu_1.
\end{equation}
In view of \eqref{blow-1-0-3}, it follows that
 \begin{equation*}
 \begin{split}
\Gamma_t+&\left(k_1(2\mu_0u-u_x^2)+k_2u\right)\Gamma_x\\
=&-2\Gamma^2+ (\fr 32 k_2-k_1\mu_0) u_x\Gamma\\
&+(k_1m+k_2)\left[k_1\left(2\mu_0^2(u-\mu_0)-\mu_0\mu_1^2\right)+k_2\left(2\mu_0(u-\mu_0)-\fr 12 \mu_1^2\right)\right]\\
\leq &\;-2\Gamma^2+ (\fr 32 k_2-k_1\mu_0) u_x\Gamma\\
&+(k_1m+k_2)\left[k_1 \mu_0\mu_1\left(\fr {\sqrt{3}}3\mu_0-\mu_1\right)+k_2\mu_1\left(\fr {\sqrt{3}}3 \mu_0-\fr 12 \mu_1\right)\right].
\end{split}
\end{equation*}
This completes the proof of Lemma \ref{lem-blow-2}.
\end{proof}

In the following,  we use the notations $\tau=\frac tb$, $b=\sqrt{\frac 1{2k_2C_2}}$, $n(\tau,x)=\frac 1{m(t,x)}$, $n_0(x)=n(0,x)$, and $\Gamma=k_1mu_x+k_2 u_x$ for convenience of the certain complicated computations.

In view of signs of $C_1$ and $C_2$, we need to consider four cases: (i). $C_1\leq 0$, $C_2\leq 0$; (ii). $C_1 \leq 0$, $C_2 > 0$; (iii). $C_1 > 0$, $C_2 \geq 0$; (iv). $C_1 > 0$, $C_2 < 0$. For the case (i), it is obvious to see from \eqref{blow-2-0} that the solution blows up at finite time. Here we only consider the case (ii). The other cases can be discussed in a similar manner. For the case (ii), we have the following blow-up result, which demonstrates that wave-breaking depends on the infimum of $\Gamma=k_1mu_x+k_2 u_x$ for time $ t \in [0, T). $
\begin{thm}\label{thm-7.2}
Suppose $k_1>0$, $k_2>0$, $C_1\leq 0$, $C_2>0$ and $u_0\in H^s$, with $s > 5/2$. Let $ T^*
> 0 $ be the maximal time of existence of the
 the corresponding solution $m(t,x)$ of \eqref{cauchy} with the initial value $ m_0$. Assume
 $m_0(x)=(\mu-\partial_x^2)u_0\geq0$ for any $x\in{\mathbb S }$ and $m_0(x_0)>0$ for some $x_0\in \Bbb S$ and that
% \begin{eqnarray}
\begin{subequations}
\begin{equation}\label{blow-data-1a}
 (1) \quad \qquad  \qquad \qquad \qquad\qquad \qquad  u_{0,x}(x_0)=-\fr 1 {2bk_2}, \qquad {\it or} \qquad  \qquad
\end{equation}
\begin{equation}\label{blow-data-1b}
(2) \qquad \qquad\qquad  u_{0,x}(x_0)< \min \left \{-\frac 1{2bk_2}, -\frac {\sqrt{k_2+2 k_1m_0(x_0)}}{2b\sqrt{k_2}(k_1m_0(x_0)+k_2)} \right \}.
\end{equation}
\end{subequations}
 % \end{eqnarray}
 Then the solution blows up at finite time $ T^* \leq t^* $, where
\begin{eqnarray*}
\begin{aligned}
&t^*=\left\{\begin{array}{c}
\sqrt{\fr 1{2k_2C_2}}\ln \left(\frac {k_2}{k_1m_0(x_0)}+1\right), \qquad  {\rm for} \quad  {\rm (1)},\\[.5cm]
\sqrt{\fr 1{2k_2C_2}}\ln \left(\frac {\fr {k_1}{k_2}-\sqrt{n_{\tau}^2(0,x_0)-n_0^2(x_0)-2\fr {k_1}{k_2}n_0(x_0)}}{n_0(x_0)+n_{\tau}(0,x_0)+\fr{k_1}{k_2}}\right), \qquad  \;\; {\rm for} \quad {\rm (2)}\\
\end{array}\right.
\end{aligned}
\end{eqnarray*}
with
\begin{eqnarray*}
\begin{aligned}
n_{\tau}(0,x_0)=2b\big(k_1+\fr{k_2}{m_0(x_0)}\big)u_{0,x}(x_0).
\end{aligned}
\end{eqnarray*}
\end{thm}
\begin{proof} Again, by  Theorem \ref{thm-blow-criterion-2} with the density argument, it suffices
to consider the case $s\geq3$. Recall $\Gamma=k_1mu_x+k_2u_x$. Thanks to \eqref{flow-1} and \eqref{blow-1-0-3}, along the flow in \eqref{flow-1} we have
\begin{equation}\label{eq-7.18}
 \begin{split}\Gamma_{t}+\left(k_1(2\mu_0u-u_x^2)+k_2u\right)\Gamma_x\leq -2\Gamma^2+C_1 u_x\Gamma +C_2 (k_1m+k_2).
 \end{split}
\end{equation}
Similarly, one can see from \eqref{cauchy} that
 \begin{equation*}\frac{dm}{dt}(t,q(t,x_0))=-2m\Gamma(t,q(t,x_0)).
  \end{equation*}
Notice that $3k_2-2\mu_0k_1 \leq 0$, which implies that
\begin{equation}\label{eq-7.19}
 \begin{split}
-\frac{d^2}{dt^2}n(\tau,q(t,x_0))=&\frac d{dt}(\frac1{m^2}\frac{dm(\tau,q(t,x_0))}{dt})=-2\frac d{dt}(\frac \Gamma m)(\tau,q(t,x_0))\\
\geq &\;\big(2k_1\mu_0-3k_2\big)u_x^2\big(k_1+\fr {k_2}m\big)\\
&\;-2\big[k_1\mu_0\mu_1(\frac{\sqrt{3}}3\mu_0-\mu_1)+k_2\mu_1(\frac{\sqrt{3}}3\mu_0-\frac12\mu_1)\big](k_1+\frac1mk_2),\\
\geq&\; -2C_2(k_1+k_2n)(\tau,q(t,x_0)).
 \end{split}
\end{equation} It is found from  (\ref{eq-7.19}) that
\begin{equation}\label{eq-7.20}
(1-\frac{d^2}{d\tau^2})n(\tau,q(t,x_0)) \geq-\frac{k_1}{k_2}.
\end{equation}
Multiplying (\ref{eq-7.20}) by $e^{\tau}$ and integrating the resulting equation from $0$ to
$\tau$, we obtain
\begin{equation}\label{eq-7.21}
e^{\tau}\big(n(\tau,q(t,x_0))-n_\tau(\tau,q(t,x_0))\big)\geq n_0(x_0)-n_{\tau}(0,x_0)-\frac{k_1}{k_2}(e^\tau-1).
\end{equation}
Furthermore, multiplying (\ref{eq-7.21}) by $e^{-2\tau}$ and integrating it from $0$ to $\tau$ yields
\begin{equation}\label{eq-7.22}
 \begin{split}
0\leq&\; n(\tau,q(t,x_0))\\
 \leq&\; \frac12 e^\tau\Big(n_0(x_0)+n_{\tau}(0,x_0)+\frac{k_1}{k_2}\Big)-\frac{k_1}{k_2}
+\frac12e^{-\tau}\Big(n_0(x_0)-n_{\tau}(0,x_0)+\frac{k_1}{k_2}\Big).
 \end{split}
\end{equation}
In view of the assumptions of the theorem, we now consider two possibilities.

\noindent (1). If $x_0$ satisfies \eqref{blow-data-1a}, then we have
\begin{eqnarray*}
\begin{aligned}
n_0(x_0)&\;+n_{\tau}(0,x_0)+\fr {k_1}{k_2}=\fr 1{m_0(x_0)}+\fr {2b\Gamma(0,x_0)}{m_0(x_0)}+\fr {k_1}{k_2}\\
=&\;\big(1+2bk_2 u_{0,x}(x_0)\big)\big(\fr {k_1}{k_2}+\fr 1{m_0(x_0)}\big)=0.
\end{aligned}
\end{eqnarray*}
It implies that
\begin{eqnarray*}
\begin{aligned}
n_0(x_0)&\;-n_{\tau}(0,x_0)+\fr {k_1}{k_2}=\fr 2{m_0(x_0)}+\fr {2k_1}{k_2}.
\end{aligned}
\end{eqnarray*}
From \eqref{eq-7.22} and above expressions, we have
\begin{eqnarray*}
0\leq \frac 1{m(t,q(t,x_0))} \leq -e^{-\tau}\left(\fr{k_1}{k_2}e^{\tau}-\fr 1{m_0(x_0)}-\fr {k_1}{k_2}\right).
\end{eqnarray*}
It implies that one may find a time $0<T^*\leq t^*=\sqrt{\fr 1{2C_2k_2}}\ln (1+\fr {k_2}{k_1m_0(x_0)})$ such that
 \begin{eqnarray*}
 m(t,q(t,x_0)) \rightarrow +\infty, \;\;{\rm as} \;\; t\rightarrow T^*.
\end{eqnarray*}
In terms of \eqref{eq-7.21}, we also have
\begin{eqnarray}\label{eq-7.23}
\begin{aligned}
\Gamma(t,q(t,x_0))=&\;-\frac{m_t(t,q(t,x_0))}{2m(t,q(t,x_0))}=\frac 1{2b}\frac {n_{\tau}(\tau,q(t,x_0))}{n(\tau,q(t,x_0))}\\
\leq &\; \frac 1{2bn(\tau,q(t,x_0))}\left[n(\tau,q(t,x_0))+e^{-\tau}\left(\fr {k_1}{k_2}(e^{\tau}-1)-n_0(x_0)+n_{\tau}(0,x_0) \right)\right]\\
= &\; \frac 1{2b}\left[1+e^{-\tau}\left(\fr {k_1}{k_2}(e^{\tau}-1)-n_0(x_0)+n_{\tau}(0,x_0) \right)m(t,q(t,x_0))\right]\\
=&\; \frac 1{2b}\left[1+e^{-\tau}\left(\fr {k_1}{k_2}e^{\tau}-2n_0(x_0)-\fr {2k_1}{k_2} \right)m(t,q(t,x_0))\right].
\end{aligned}
\end{eqnarray}
As discussed above, we see that
\begin{eqnarray*}
\begin{aligned}
\fr {k_1}{k_2}e^{\tau}-2n_0(x_0)-\fr {2k_1}{k_2}\leq -\big(n_0(x_0)+\fr {k_1}{k_2} \big), \;\; {\rm for }\;\; t \leq T^*.
\end{aligned}
\end{eqnarray*}
Thus we deduce from \eqref{eq-7.23} that
\begin{eqnarray*}
\begin{aligned}
\liminf\limits_{t\uparrow T^*}\big(\inf\limits_{x\in\mathbb{S}}\Gamma(t,x)\big)=-\infty \;\; {\rm as}\;\; t\rightarrow T^*.
\end{aligned}
\end{eqnarray*}
(2). If the initial data satisfies \eqref{blow-data-1b}, it then  follows that
\begin{eqnarray*}
\begin{aligned}
n_0(x_0)+&\;n_{\tau}(0,x_0)+\fr {k_1}{k_2}=\;\big(1+2bk_2u_{0,x}(x_0)\big)(\fr {k_1}{k_2}+\fr 1{m_0(x_0)})<0,
\end{aligned}
\end{eqnarray*}
and
\begin{eqnarray*}
\begin{aligned}
4b^2(k_1 m_0(x_0)+k_2)^2u_{0,x}^2(x_0)>1+\fr {2k_1}{k_2}m_0(x_0).
\end{aligned}
\end{eqnarray*}
A direct computation yields
\begin{eqnarray*}
n_{\tau}(0,x_0)=&\;\frac {2b\Gamma(0,x_0)}{m_0(x_0)}=2b(k_1+ \fr{k_2}{m(0,x_0)})u_{0,x}(x_0).
\end{eqnarray*}
Combining above expressions, we find
\begin{eqnarray*}
n_{\tau}^2(0,x_0)-n_0^2(x_0)-2\frac{k_1}{k_2}n_0(x_0)>0.
\end{eqnarray*}
and
\begin{eqnarray*}
\sqrt{n_{\tau}^2(0,x_0)-n_0^2(x_0)-2\frac{k_1}{k_2}n_0(x_0)}>-n_0(x_0)-n_{\tau}(0,x_0).
\end{eqnarray*}
So the equation
 \begin{eqnarray}\label{eq-7.24}
 \frac12 \Big(n_0(x_0)+n_{\tau}(0,x_0)+\fr {k_1}{k_2}\Big)z^2-\fr{k_1}{k_2}z
+\frac12 \Big(n_0(x_0)-n_{\tau}(0,x_0)+\fr{k_1}{k_2}\Big)=0
\end{eqnarray}
has one positive root
\begin{eqnarray*}\label{root-1}
z_1=\frac{\fr{k_1}{k_2}-\sqrt{n_{\tau}^2(0,x_0)-n_0^2(x_0)-2\frac{k_1}{k_2}n_0(x_0)}}{n_0(x_0)+n_{\tau}(0,x_0)+\frac{k_1}{k_2}} \end{eqnarray*}
and one negative root
\begin{eqnarray*}\label{root-1}
z_2=\frac{\fr{k_1}{k_2}+\sqrt{n_{\tau}^2(0,x_0)-n_0^2(x_0)-2\frac{k_1}{k_2}n_0(x_0)}}{n_0(x_0)+n_{\tau}(0,x_0)+\frac{k_1}{k_2}}. \end{eqnarray*}
It is then inferred from \eqref{eq-7.22} that
\begin{eqnarray*}
0\leq \frac 1{m(t,q(t,x_0))}\leq \frac12 e^{-\tau} \big(n_0(x_0)+n_{\tau}(0,x_0)+\fr {k_1}{k_2}\big) (e^{\tau}-z_1)(e^{\tau}-z_2),
\end{eqnarray*}
Note that $z_1>1$ and $z_2<0$.  Then one may find a time $0<T^*\leq t^*=b\ln z_1$ such that
\begin{eqnarray*}
m(t)\rightarrow +\infty \;\; {\rm as} \;\; t \rightarrow T^*.
\end{eqnarray*}
For any $t\leq T^*$, we have
\begin{eqnarray*}
\begin{aligned}
n_{\tau}(0,x_0)&\;+\fr {k_1}{k_2}(e^{\tau}-1)-n_0(x_0)\\
\leq &\;\fr {k_1}{k_2}\frac {\fr {k_1}{k_2}-\sqrt{n_{\tau}^2(0,x_0)-n_0^2(x_0)-2\frac{k_1}{k_2}n_0(x_0)}}{n_0(x_0)+n_{\tau}(0,x_0)+\frac{k_1}{k_2}}
-n_0(x_0)+n_{\tau}(0,x_0)-\frac{k_1}{k_2}\\
=&\; \frac{\sqrt{n_{\tau}^2(0,x_0)-n_0^2(x_0)-2\frac{k_1}{k_2}n_0(x_0)}\big(\sqrt{n_{\tau}^2(0,x_0)-n_0^2(x_0)-2\frac{k_1}{k_2}n_0(x_0)}-\fr {k_1}{k_2}\big)}{n_{\tau}(0,x_0)+n_0(x_0)+\fr{k_1}{k_2}}  < 0.
\end{aligned}
\end{eqnarray*}
Making use of \eqref{eq-7.23}, we deduce that
\begin{eqnarray*}
\begin{aligned}
\liminf\limits_{t\uparrow T^*}\left (\inf\limits_{x\in\mathbb{S}}\Gamma(t,x)\right )=-\infty \;\; {\rm as}\;\; t\rightarrow T^*.
\end{aligned}
\end{eqnarray*}
This completes the proof of the theorem.
 \end{proof}

\begin{remark} In the case of $C_1 > 0$, $C_2 \geq 0$, replacing $C_2$ by
\begin{eqnarray*}
\tilde{C}_2=(3k_2-2\mu_0 k_1)\mu_0^2+\frac {5\sqrt{3}}{3}k_2 \mu_0\mu_1-\frac 1{12}(9k_2+26k_1\mu_0)\mu_1^2,
\end{eqnarray*}
Theorem \ref{thm-7.2} still holds.
\end{remark}

\begin{remark} In the case of $C_1 > 0$ and $C_2 < 0$, Theorem \ref{thm-7.2}  holds  just by changing $C_2$ to
\begin{eqnarray*}
\tilde{C}_2=(3k_2-2\mu_0 k_1)(\mu_0+\fr {\sqrt{3}}{6}\mu_1)^2.
\end{eqnarray*}
\end{remark}

Recall that $\mu_0=\mu(u(t)) = \int_{\mathbb{S}}u(t,x)\ dx $ and $
m_0=(\mu-\partial_x^2)u_0$.

\begin{lem} {\cite{ce-1}} \label{lem-7.1}
Let $T>0$ and $v\in C^1([0,T);H^2(\mathbb{R})).$ Then for
every $t\in[0,T),$ there exists at least one point
$\xi(t)\in{\mathbb{R}}$ with
$$
I(t):= \underset {x\in
\mathbb{R}}\inf\left(v_x(t,x)\right)=v_x(t,\xi(t)).$$ The function
$I(t)$ is absolutely continuous on $(0,T)$ with
 \begin{equation*}\frac{dI(t)}{dt}=v_{tx}(t,\xi(t)),   \; a. \,  e. \, \mbox{on} \; (0,T).
  \end{equation*}
\end{lem}
The following result demonstrates that wave-breaking could depend on the infimum of $u_x$.
\begin{thm} \label{7-blowup-3}
Assume $k_1 >0$, $k_2 >0$ and $m_0>0$. Let $u_0 \in H^s,  s > 5/2$ and $T^*>0$ be the maximal time of
existence of the corresponding solution $u(t,x)$ to (\ref{cauchy})
with the initial data $u_0$. If $C_2>0$ (defined in Lemma \ref{lem-blow-2})and
$$ \inf_{x \in \mathbb{S}}  (u_{0,x}(x)) <-\sqrt{\dfrac{2k_1\mu_0\mu_1(\sqrt{3}\mu_0-3\mu_1)+k_2\mu_1(2\sqrt{3}\mu_0-3\mu_1)}
{3k_2+6\mu_0k_1}}:\equiv-K,$$
then the corresponding
solution $u(t,x)$ to (\ref{cauchy}) blows up in finite time  $T^*$ with
$$0<T^*\le -\dfrac{4}{(k_2+2k_1\mu_0)(\inf\limits_{x \in\mathbb{S}} u_{0,x}(x)+\sqrt{-K\inf\limits_{x \in \mathbb{S}} u_{0,x}(x)})},$$
such that
$$\liminf\limits_{t\uparrow T^*}\left (\inf\limits_{x\in\mathbb{S}}(u_x(t,x))\right )=-\infty.$$
\end{thm}
\begin{proof}
The proof of the theorem is approached by applying Lemma \ref{lem-7.1}.
In fact, if we define $w(t)=u_x(t,\xi(t))=\inf\limits_{x\in\mathbb{S}}[u_x(t,x)]$,
then for all $t\in[0,T^*)$, $u_{xx}(t,\xi(t))=0$. Thus one finds that
$$
\frac{d}{dt} w(t)\le-(\frac1{2}k_2+k_1\mu_0)\big(w^2(t)-K^2\big).$$ Applying the assumptions of
Theorem \ref{7-blowup-3}, it can be easily deduced that if
$$w(0)<-\sqrt{\dfrac{2k_1\mu_0\mu_1(\sqrt{3}\mu_0-3\mu_1)+k_2\mu_1(2\sqrt{3}\mu_0-3\mu_1)}
{3k_2+6\mu_0k_1}},$$ then $T^*$ is
finite and $\liminf\limits_{t\uparrow T^*}\left(\inf\limits_{x\in\mathbb{S}}u_x(t,x)\right)=-\infty$.
This completes the proof of the theorem.
\end{proof}
The following result shows that wave-breaking may depend on the infimum of $mu_x$.
\begin{thm}\label{7-thm-blowup-result-1}
Suppose $u_0\in H^s({\Bbb S}),\, s > 5/2$ with $k_2>0,\;k_1>0$ and
$ C_2>0$ (defined in Lemma \ref{lem-blow-2}). Let $T^*>0$ be the maximal time of existence of the corresponding solution $u(t,x)$ to
\eqref{cauchy} with the initial data $u_0$. Assume
that for some $x_0 \in {\Bbb S}$ with $m_0(x_0)>0$ and
\begin{equation}\label{blow-4-0}
 u_{0,x}(x_0) <-\frac{\sqrt{C_2(k_1m_0(x_0)+k_2)}}{
 k_1m_0(x_0)}.
\end{equation}
Then the solution $u(t, x)$ blows up at finite time
 \begin{equation}\label{blow-4-1}
T^* \leq t^{\ast}=-\frac{k_1m_0(x_0)u_{0,x}(x_0)}{C_2(k_1m_0(x_0)+k_2)}
-\sqrt{\left(\frac{k_1m_0(x_0)u_{0,x}(x_0)}
{C_2(k_1m_0(x_0)+k_2)}\right)^2-\frac{1}{C_2(k_1m_0(x_0)+k_2)}}.
\end{equation}
Moreover, when $T^*=t^{\ast}$, the following estimate of the blow-up
rate holds
\begin{equation}\label{blow-4-1-a}
\begin{split}
&\liminf_{t\uparrow T^{*}}\left((T^*-t)\,\inf_{x\in {\Bbb S}}mu_x(t,
x)\right) \leq -\frac1{2k_1}.
\end{split}
\end{equation}
\end{thm}

\begin{proof} As in the proof of Theorem \ref{thm-blow-criterion-2}, it suffices
to consider the case $s\geq3$.
Thanks to \eqref{blow-1-0-1} and \eqref{flow-1}, we have
 \begin{equation*}
\frac{d}{dt}M(t, q(t, x_0))\leq -2k_1 M^2(t, q(t, x_0))+C_2 m(t,q(t, x_0)).
\end{equation*}
And recall that
 \begin{equation*}
\frac{d}{dt}m(t, q(t, x_0))= -2M(t, q(t, x_0))(k_1m+k_2).
\end{equation*}
Denote $\overline{w}(t):=2\big(m(t,q(t, x_0))+k_2/k_1\big)$, then
 \begin{align*}
\frac{d}{dt}M(t, q(t, x_0))&\leq -2k_1 M^2(t, q(t, x_0))+C_2 \left(\frac{\overline{w}}{2}-\frac{k_2}{k_1}\right)\\
&\leq -2k_1 M^2(t, q(t, x_0))+ \frac{C_2}{2}\overline{w}.
\end{align*}
Applying the above two equations, we deduce that
\begin{align*}
\frac{d}{dt}\left(\frac{1}{\overline{w}(t)^2}\right.\left.\frac{d}{dt}\overline{w}(t)\right)
&=-2k_1\frac{d}{dt}\left(\frac{1}{\overline{w}(t)}M\right)\\
&=\frac{2k_1}{\overline{w}(t)^2}\left(-\overline{w}(t)\frac{d}{dt}M(t)
+M(t)\frac{d}{dt}\overline{w}(t)\right)\geq -k_1 C_2.
\end{align*}
Integrating it from $0$ to $t$ leads to
\begin{equation*}
\frac{1}{\overline{w}(t)^2}\frac{d}{dt}{\overline{w}(t)} \geq
C_3-k_1 C_2t,
\end{equation*}
with
\begin{equation*}
C_3:=\frac{\overline{w}(0)^{'}}{\overline{w}(0)^2}=-\frac{2k_1M(0)}{\overline{w}(0)}
=-\dfrac{k_1^2(m_0u_{0,x})(x_0)}{k_1m_0(x_0)+k_2}.
\end{equation*}
and hence
\begin{align*}
\frac{1}{\overline{w}(t)}&\leq
\frac1{2}\left(k_1C_2t^2-2C_3t+\frac{k_1}{k_1m_0(x_0)+k_2}\right).
\end{align*}
Let \begin{equation*}
t^{\ast}:=\frac{C_3}{k_1C_2}-
\sqrt{\left (\frac{C_3}{k_1C_2} \right
)^2-\frac{1}{C_1(k_1m_0(x_0)+k_2)}}
\end{equation*}
and
\begin{equation*}
 t_{\ast}:=\frac{C_3}{k_1C_2}+\sqrt{\left (\frac{C_3}{k_1C_2} \right
)^2-\frac{1}{C_2(k_1m_0(x_0)+k_2)}}\ .
\end{equation*}
Thus,
\begin{equation*}
0 \leq \frac{1}{\overline{w}(t)}\leq \frac1{2}k_1C_2(t-t^{\ast})(t-t_{\ast}).
\end{equation*}
Therefore,
\begin{equation*}
\inf_{x\in {\Bbb S}} mu_x(t, x) \leq M(t)\longrightarrow  -\infty, \quad
\mbox{as} \quad t \longrightarrow  T^* \leq t^{\ast},
\end{equation*}
This implies that the solution $u(t,x)$ blows up at the time $T^*$. \end{proof}

\noindent {\bf Acknowledgments.} The work of Fu is partially
  supported by the NSF-China grant-11001219
and the Scientific Research Program Funded by Shaanxi Provincial
Education Department grant-2010JK860. The work of Liu is partially
supported by the NSF grant DMS-1207840. The work of Qu is supported in part by the
NSF-China for Distinguished Young Scholars grant-10925104.

\end{document}